\documentclass[reqno,12pt]{amsart}

\usepackage{graphicx}         
\usepackage{amsfonts}
\usepackage{amssymb}
\usepackage[latin1]{inputenc}
\usepackage[all]{xy}
\usepackage{tikz-cd}        
\usepackage{pgf,tikz}

\newtheorem{theorem}{Theorem}[section]
\newtheorem{corollary}[theorem]{Corollary}
\newtheorem{lemma}[theorem]{Lemma}

\newtheorem{proposition}[theorem]{Proposition}
\theoremstyle{definition}
\newtheorem{definition}[theorem]{Definition}
\theoremstyle{remark}

\newtheorem{remark}[theorem]{Remark}
\newtheorem{example}[theorem]{Example}

\renewcommand{\l}{\lambda}

\newcommand{\pb}[1]{\left\{#1\right\}}

\newcommand{\lb}[1]{\[#1\]}

\newcommand{\gb}[1]{[\![#1]\!]}
\newcommand{\GB}{[\![\cdot\,,\cdot]\!]}

\renewcommand{\(}{\left(}
\renewcommand{\)}{\right)}
\renewcommand{\[}{\left[}
\renewcommand{\]}{\right]}

\newcommand{\set}[1]{\left\{#1\right\}}

\newcommand{\cat}[3]{#1^{\hbox{\tiny #2}}_{\hbox{\tiny $#3$}}}

\newcommand{\catns}{\cat{\cC}{NS}{}}
\newcommand{\catnij}{\cat{\cC}{Nij}{}}
\newcommand{\cattrb}{\cat{\cC}{TRB}{}}
\newcommand{\catrrbl}{\cat{\cC}{RRB}{\lambda}}

\newcommand{\cC}{\mathcal C}

\newcommand{\cO}{\mathcal O}
\newcommand{\cR}{\mathcal R}

\newcommand{\bbN}{\mathbb N}

\newcommand{\fR}{\mathfrak R}
\newcommand{\fN}{\mathfrak N}
\newcommand{\op}{\beta}
\newcommand{\po}{\gamma}
\newcommand{\wee}{.}
\newcommand{\fT}{{\rm T}}

\newcommand{\tsar}{\rotatebox[origin=c]{180}{$\,\star\,$}}
\newcommand{\pr}{\bullet}
\newcommand{\aspr}{\ast_\pr}
\newcommand{\pow}[1]{{}^{[#1]}}

\newcommand{\bt}{\boxtimes}
\newcommand{\ul}[2]{\underline{#1}_{#2}}
\renewcommand{\a}[1]{\ul a{#1}}
\renewcommand{\b}[1]{\ul b{#1}}
\renewcommand{\c}[1]{\ul c{#1}}
\newcommand{\m}[1]{\ul m{#1}}
\renewcommand{\mp}[1]{\ul {m'}{#1}}

\newcommand{\Gr}{\mathop{\rm Gr}\nolimits}

\newcommand{\leqs}{\leqslant}
\newcommand{\geqs}{\geqslant}

\newcommand{\End}{{\mathop{\rm End}}}
\newcommand{\Hom}{{\mathop{\rm Hom}}}
\newcommand{\Id}{\mathrm{Id}}

\allowdisplaybreaks[4]

\makeatletter
\@namedef{subjclassname@2020}{\textup{2020} Mathematics Subject Classification}
\makeatother

\newif\ifprivate
\privatefalse

 \numberwithin{equation}{section}

\def\???{\ifprivate {\bf {???}} \marginpar{{\Huge {\bf ?}}}\else \fi}
\numberwithin{equation}{section}

\begin{document}  
	
	\nocite{*}
	
	\title[NS-algebras]{Generalized NS-algebras}
	\author[Ospel]{Cyrille Ospel} \address{Cyrille Ospel, LaSIE, UMR 7356 CNRS, Universit\'e de La Rochelle,
		Av. M. Cr\'epeau, 17042 La Rochelle cedex 1, La Rochelle, France}\email{cospel@univ-lr.fr}
	
	\author[Panaite]{Florin Panaite}
	\address{Florin Panaite, Institute of Mathematics of the Romanian Academy, PO-Box 1-764, RO-014700 Bucharest, Romania}
	\email{florin.panaite@imar.ro}
	
	\author[Vanhaecke]{Pol Vanhaecke}
	\address{Pol Vanhaecke, Universit\'e de Poitiers, Laboratoire de Math\'ematiques et Applications, B\^at.\ H3 - Site du
		Futuroscope, TSA 61125, 11 bd Marie et Pierre Curie,
		86 073 Poitiers Cedex 9, France}\email{pol.vanhaecke@math.univ-poitiers.fr}
	
	
	\date{\today}
	
	\begin{abstract}
		We generalize the notion of an NS-algebra, which was previously only considered for associative, Lie and Leibniz
		algebras, to arbitrary categories of binary algebras with one operation. We do this by defining these algebras
		using a bimodule property, as we did in our earlier work for defining the notions of a dendriform and tridendriform
		algebra for such categories of algebras. We show that several types of operators lead to NS-algebras: Nijenhuis
		operators, twisted Rota-Baxter operators and relative Rota-Baxter operators of arbitrary weight. We thus provide a
		general framework in which several known results and constructions for associative, Lie and Leibniz-NS-algebras are
		unified, along with some new examples and constructions that we also present.
		
	\end{abstract}
	
	\maketitle
	
	\setcounter{tocdepth}{1}
	
	\tableofcontents
	
	\section{Introduction}
	
	A (classical) NS-algebra is an algebra with three products whose sum is an associative product, so these products
	``split'' associativity. NS-algebras are in that sense alternatives to tridendriform algebras and to dendriform
	algebras, where the products on the latter split the associativity in two, rather than three, products. Motivated
	by algebraic K-theory and algebraic topology, Loday \cite{lodaydend} introduced a dendriform algebra as being an
	algebra $(A,\prec,\succ)$, for which the two products $\prec$ and $\succ$ satisfy for all $a,b,c\in A$,
	the~relations
	\begin{align}\label{equ:dendri_def}
		(a\prec b)\prec c=a\prec (b\star c)\;,\quad
		(a \succ b)\prec c=a\succ (b\prec c)\;, \quad
		(a\star b)\succ c=a\succ (b\succ c)\;,
	\end{align}
	where $a\star b:=a\prec b+a\succ b$. Summing up these three relations, it is clear that $(A,\star)$ is an
	associative algebra; on it, the dendriform products provide some extra structure. A similar but more general
	concept, called tridendriform algebra, was introduced by Loday and Ronco in~\cite{lodayronco}. Both definitions
	were afterwards also adapted to other types of binary algebras, for example Lie, Leibniz, pre-Lie and Jordan
	algebras, see for instance \cite{baiLdendri,Jdendri,prejordan,shengtang}.  These definitions were generalized to
	arbitrary categories of binary algebras satisfying a given set of multilinear relations. This was done in an
	operadic context and language in \cite{baibellier,GK,pei}, while we proposed in \cite{opv} an equivalent
	generalization, based on a natural bimodule property.  For comparison with the definition of a generalized
	NS-algebra that we will propose, we recall from~\cite{opv} the definition of a $\cC$-dendriform algebra, where
	$\cC$ denotes the category of all binary algebras $(A,\mu)$ which satisfy a given set of multilinear relations
	$\cR_1=0,\dots,\cR_k=0$.
	\begin{definition}\label{def:dendri_gen}
		An algebra $(A,\prec,\succ)$ is said to be a \emph{$\cC$-dendriform algebra} if \hbox{$(A\oplus A,\boxtimes)\in\cC$},
		where $\boxtimes$ is defined, for $(a,x),(b,y)\in A\oplus A$, by
		\begin{equation}\label{eq:semi-direct}
			(a,x)\boxtimes (b,y):=(a\star b,a\succ y+x\prec b)\;,\quad\hbox{where}\quad a\star b:= a\succ b+a\prec b\;.
		\end{equation}%
	\end{definition}
	This is equivalent to saying: $(A,\star)\in\cC$ and $(A,\succ,\prec)$ is an $(A,\star)$-bimodule. For $\cC$ the
	category of associative algebras, one easily recovers the relations \eqref{equ:dendri_def} from this definition.
	
	The notion of a (classical) NS-algebra was introduced by Leroux~\cite{leroux}, who defined an NS-algebra as an
	algebra $(A,\prec,\succ,\vee)$ with three operations, satisfying the following 4 relations:
	\begin{gather}
		(a\prec b)\prec c=a\prec (b\star c)\;,\quad
		(a \succ b)\prec c=a\succ (b\prec c)\;, \quad
		(a\star b)\succ c=a\succ (b\succ c)\;,\nonumber\\
		(a\vee b)\prec c+(a\star b)\vee c=a\succ(b\vee c)+a\vee(b\star c)\;,\label{eq:assoc_NS}
	\end{gather}
	where $a\star b$ now stands for $a\star b:= a\succ b+a\prec b+a\vee b$. Summing up these equations, one finds
	$(a\star b)\star c=a\star(b\star c)$, so that the three products $\prec,\succ$ and $\vee$ indeed split
	associativity. The term NS-algebra was coined by Leroux in \cite{leroux} as a reminder of Nijenhuis, motivated by
	the fact that his main class of examples of NS-algebras were provided by Nijenhuis operators (on associative
	algebras). Independently, Uchino introduced in an unpublished version of \cite{uchino} (see for example
	\cite{uchino_pre}) the notion of a twisted dendriform algebra, before realizing that it coincided with the notion
	of an NS-algebra, introduced by Leroux; his main class of examples of twisted dendriform algebras were provided by
	twisted Rota-Baxter operators, explaining his choice of terminology. Uchino also obtained significant results on
	(classical) NS-algebras, just like Lei and Guo in \cite{leiguo}, who constructed free Nijenhuis algebras to obtain
	the universal enveloping Nijenhuis algebra of a classical NS-algebra.
	
	Recently, Das and collaborators contributed to the study of NS-algebras, in two directions, on the one hand on
	(classical) NS-algebras in \cite{dastwistedassoc}, on the other hand they introduced the notion of an NS-algebra
	for two other types of algebras, namely Lie algebras \cite{dastwistedLie} and Leibniz
	algebras~\cite{dasleibniz}. Upon observing that the three notions of NS-algebras correspond to a natural bimodule
	property, analogous to the one defining $\cC$-dendriform and $\cC$-tridendriform algebras, we introduce the general
	notion of a $\cC$-NS-algebra as follows:
	
	\begin{definition}\label{def:NS_intro}
		An algebra $(A,\prec,\succ,\vee)$ is said to be a \emph{$\cC$-NS-algebra} if $(A\oplus A,\boxtimes)\in\cC$,
		where~$\boxtimes$ is defined for $(a,x),(b,y)\in A\oplus A$ by
		\begin{equation}
			(a,x)\boxtimes(b,y):=(a\star b,a\succ y+x\prec b)\;,\quad\hbox{where}\quad a\star b:= a\succ
			b+a\prec b+a\vee b\;.
		\end{equation}
	\end{definition}
	Taking $\cC$ to be the category of associative algebras, it is easy to recover the relations~\eqref{eq:assoc_NS}
	defining associative-NS-algebras, and similarly for the other two known cases, namely Lie and
	Leibniz-NS-algebras. Taking any other category of binary algebras we get in the same way from the definition the
	relations defining new examples of NS-algebras, such as NAP-NS-algebras, Jordan-NS-algebra, Poisson-NS-algebras and
	so on.
	
	As in the case of $\cC$-dendriform algebras, the condition in the definition is equivalent to saying:
	$(A,\star)\in\cC$ and $(A,\succ,\prec)$ is an $(A,\star)$-bimodule, the difference being that the definition of
	$\star$ is different. This minor difference has major consequences and it turns out that $\cC$-NS-algebras are more
	closely related to $\cC$-tridendriform algebras than to $\cC$-dendriform algebras: as we will prove, every
	$\cC$-tridendriform algebra is (in a natural but non-trivial way) a $\cC$-NS-algebra. This was already observed in
	the associative case by Uchino in \cite{uchino_pre}, but is new in the Lie and Leibniz case.
	
	Rota-Baxter operators on a binary algebra (see \cite{Guo} for the associative case) lead to dendriform or
	tridendriform algebras (depending on their weight). Recall that a linear map $\fR:A\to A$ on an algebra $(A, \mu
	)\in \cC$ is said to be a \emph{Rota-Baxter operator of weight~$\l$} if
	\begin{equation}\label{equ:RB_intro}
		\fR (a)\fR(b)=\fR(\fR(a)b+a\fR (b)+\l ab)\;,
	\end{equation}
	for all $a,b\in A$. The new products defined for all $a,b\in A$ by
	\begin{equation}
		a\prec b:= a\fR(b)\;,\quad a\succ b:=\fR(a)b\;,\quad a\wee b:=\l ab\;
	\end{equation}
	make $(A,\prec,\succ,\wee)$ into a $\cC$-tridendriform algebra (this was first proved for associative algebras in
	\cite{aguiar,kef} and for arbitrary binary algebras in \cite{GK,opv}). More generally, upon properly interpreting
	\eqref{equ:RB_intro}, this construction can be generalized to \emph{relative Rota-Baxter operators} (see
	\cite{GK,pei} or Proposition \ref{prp:RBB_tridendri} below).  The latter operators were first introduced by
	Kupershmidt in the context of Lie algebras under the name of $\cO$-operators, see~\cite{Kupershmidt}. Combined with
	our above general result, we obtain that a Rota-Baxter operator of weight~$\l$ on an algebra $A$ in $\cC$, and more
	generally a relative Rota-Baxter operator on such an algebra, leads to a $\cC$-NS-algebra, where $\cC$ is again any
	category of binary algebras.
	
	
	The main interest of $\cC$-NS-algebras is that other types of operators, such as Nijenhuis and twisted Rota-Baxter
	operators, lead to such algebras, though they a priori do not lead to $\cC$-tridendriform (or $\cC$-dendriform)
	algebras. In the associative case, this was already observed by Leroux \cite{leroux} and Uchino \cite{uchino}, as
	mentioned above, and it was observed by Das and Guo \cite{dastwistedLie,dasleibniz} in the case of Lie and Leibniz
	algebras. We prove in this paper that these two types of operators yield $\cC$-NS-algebras for arbitrary categories
	of binary algebras~$\cC$, thereby unifying and generalizing these cited known results and their proofs.
	
	In fact, we will do this by proving a general result, Theorem \ref{thm:general_ns}, which with some extra work
	applies to both types of operators, and some others, thereby providing another source of examples of
	$\cC$-NS-algebras, for any $\cC$. Interestingly, the proof that one of the assumptions of the theorem is valid for
	Nijenhuis operators depends on Theorem \ref{mainNijprop}, in which we extend a classical property of Nijenhuis
	operators on Lie and associative algebras (from \cite{yvette2}, respectively~\cite{carinena}) to arbitrary
	categories of binary algebras.
	
	As a second application of Theorem \ref{mainNijprop}, we give in Proposition \ref{NijRB} a simple proof of a result
	about relative Rota-Baxter operators of weight zero, which extends known results for associative, Lie and Leibniz
	algebras, which in turn have been used to define cohomologies of relative Rota-Baxter operators of weight zero for
	those types of algebras \cite{das,controlling,tangshengzhou}, making the definition of these cohomologies also
	possible for arbitrary categories of binary algebras.
	
	An application of the theory which we develop about $\mathcal{C}$-NS-algebras concerns twisted Rota-Baxter
	operators, for which we prove in Proposition \ref{prp:NS_to_TRB} that, if $(A,\prec,\succ,\vee)$ is a
	$\cC$-NS-algebra, then the product $\vee$ gives a 2-cocycle on $A$ and the identity map is a twisted Rota-Baxter
	operator, a result which extends as well a result which Uchino \cite{uchino} proved in the associative case. It
	implies that every $\cC$-NS-algebra can be obtained from a twisted Rota-Baxter operator, again for any $\cC$.
	
	
	The structure of the paper is as follows. In Section \ref{sec1} we fix some conventions and notations, and we
	recall the basic facts that we will use about bimodules, bimodule algebras and cocycles for arbitrary binary
	algebras. In Section \ref{sec:ns_algebras} we define the general notion of an NS-algebra and show how the relations
	satisfied by a $\cC$-NS-algebra are computed, a technique which we illustrate on some examples and which we will
	use in several of our proofs. We also prove a theorem which shows that a general class of operators leads to a
	$\cC$-NS-algebra. Sections \ref{sec3}, \ref{sec:RRB} and \ref{sec2} are respectively devoted to the particular
	instances of Nijenhuis, relative Rota-Baxter and twisted Rota-Baxter operators. It is shown in these sections that
	each of these operators lead to a $\cC$-NS-algebras and it is shown in the last section that conversely every
	$\cC$-NS-algebra can be obtained from a twisted Rota-Baxter operator.

	\section{Preliminaries}\label{sec1}
	
	In this section we fix some notations and we recall a few definitions and results which will be used throughout the
	paper. We also prove an elementary proposition (Proposition~\ref{warmup}), saying that any bimodule algebra is a
	bimodule, a handy result which we will use several times; we could not find this result in the literature, even in
	the associative case, where the definition is redundant, as being a bimodule is listed as one of the axioms of a
	bimodule algebra.
	
	All algebraic structures are defined over a fixed commutative ring $R$; we write $\otimes$ for $\otimes_R$. By an
	\emph{algebra} we mean an $(s+1)$-tuple $(A,\mu_1,\dots,\mu_s)$, where $A$ is an $R$-module and each
	$\mu_i:A\otimes A\rightarrow A$ is a linear map, also called a \emph{product}. Notice that there may be several
	products, but that they all are binary. An \emph{algebra homomorphism} between two algebras $(A,\mu_1,\dots,\mu_s)$
	and $(A',\mu_1',\dots,\mu_s')$ is a linear map $f:A\to A'$ such that $f(\mu_i(a\otimes b))=\mu_i'(f(a)\otimes
	f(b))$ for all $a,b\in A$ and $1\leqs i\leqs s$.  Unless otherwise specified, for an algebra $(A,\mu)$ with a
	single product we write $ab$ for $\mu(a\otimes b)$.
	
	The very general framework in which we work is the following. We are given a fixed collection of multilinear
	relations $\cR_1=0,\dots,\cR_k=0$, which are satisfied by every algebra $(A,\mu)$ that we consider. For example,
	when the relations that we consider are anticommutativity and the Jacobi identity, then the algebras we consider
	are Lie algebras. We denote by $\cC_\cR$, or more simply $\cC$, the category of all algebras $(A,\mu)$ satisfying
	the relations $\cR_1=0,\dots,\cR_k=0$; in the above example, $\cC$ is the category of Lie algebras (over $R$). When
	we are dealing with other products on $A$, we may add the product in the notation; for a relation $\cR=0$, which
	stands for $\cR_\mu=0$, we may write for example $\cR_\star=0$ for the same relation but with $\star$ as product
	(on the same $R$-module $A$). In order to make its arguments explicit, we sometimes write $\cR(a_1,\dots,a_n)=0$
	when $\cR=0$ is an $n$-linear relation.
	
	A \emph{subalgebra} of $(A,\mu)\in\cC$ is an $R$-submodule $A'$ of $A$ such that $\mu(a\otimes b)\in A'$ for all
	$a, b\in A'$. If $\mu ':A'\otimes A'\rightarrow A'$ is the restriction of $\mu $ to $A'\otimes A'$, then clearly
	$(A',\mu')\in \cC$.
	
	Let $(A,\mu)\in\cC$ and let $M$ be an $R$-module, equipped with two linear maps $l,r:A\rightarrow \End(M)$, which we
	call \emph{left} and \emph{right actions (of $A$ on $M$)}. We will simply write $a\cdot m$ for $l(a)(m)$ and
	$m\cdot a$ for $r(a)(m)$.  On $A\oplus M$ we define a product $*$ by setting
	\begin{equation}\label{equ:bimodule}
		(a,m)*(a',m'):=(aa', a\cdot m'+m\cdot a')\;,
	\end{equation}
	for all $a, a'\in A$ and $m,m'\in M$. We denote the algebra $(A\oplus M,*)$ by $A\oplus_0M$. The triplet $(M,l,r)$,
	or simply $M$, is said to be an \emph{$A$-bimodule} if $A\oplus_0M\in\cC$ (see \cite{schafer}).  The algebra
	$A\oplus_0M$ is called the \emph{trivial extension of $A$ by $M$}. Notice that we write the elements of $A\oplus M$
	as couples; we do this for the readability of the formulas.
	
	If $M$ is an $R$-module which is equipped with a product $\pr$, as well as with a left and a right action of $A$,
	we say that $(M,\pr,l,r)$, or simply $(M,\pr)$, is an \emph{$A$-bimodule algebra} if $(A\oplus M,\aspr)\in \cC$,
	where the product~$\aspr$ is defined for all $a, a'\in A$ and $m,m'\in M$ by
	\begin{equation}\label{equ:bimodule_algebra}
		(a,m)\aspr(a',m'):=(aa', a\cdot m'+m\cdot a'+m\pr m')\;.
	\end{equation}
	We call the algebra $(A\oplus M,\aspr)$ the \emph{semidirect sum} of $A$ and $M$, and we denote it by
	$A\bowtie M$. For~$\cC$ being the category of associative algebras, this definition is equivalent to Definition~2.3
	in~\cite{baipacific}; for $\cC$ being the category of Lie algebras, it is equivalent to Definition 2.1 (iii)
	in~\cite{nonabelian}.
	
	Suppose that $M$ is an $A$-bimodule and $H:A\otimes A\rightarrow M$ is a linear map. Consider on $A\oplus M$ the
	product, defined by
	\begin{equation*}
		(a,m)*_H(a',m')=(aa',a\cdot m'+m\cdot a'+H(a\otimes a'))\;,
	\end{equation*}
	for all $a, a'\in A$ and $m, m'\in M$. If $ A\oplus _HM:=(A\oplus M,*_H)\in\cC$, we say that $A\oplus _HM$ is an
	\emph{extension of~$A$ by $M$} and we call $H$ a \emph{2-cocycle (on $A$ with values in $M$)}. Even if we will not
	use this fact in this paper, let us mention that any linear combination of 2-cocycles is again a 2-cocycle. It is
	well-known that for $\cC$ being the category of associative algebras (respectively Lie algebras), $A\oplus_HM$ is
	an extension of $A$ by $M$ if and only if $H$ is a Hochschild 2-cocycle (respectively a 2-cocycle in the
	Chevalley-Eilenberg cohomology of Lie algebras).
	
	Clearly, an $A$-bimodule algebra $(M,\pr)$ for which $\pr$ is the trivial (zero) product is just an $A$-bimodule
	$M$. Also, if $M$ is an $A$-bimodule, then $A\oplus_0M$ is an extension of $A$ by $M$, with $H=0$, which explains
	the notation and terminology. Notice that taking $a=a'=0$ in \eqref{equ:bimodule_algebra}, one finds that $\pr$
	satisfies all relations satisfied by $\aspr$, hence that $(M,\pr)\in \cC$. It follows that both terms of a
	semidirect sum $A\bowtie M$ are objects in~$\cC$.
	
	\begin{example}\label{exa:A_bimodule_algebra}
		Every algebra $(A,\mu)\in\cC$ is an $A$-bimodule algebra in a natural way, namely by taking left and right
		multiplication in $A$ as left and right actions of $A$ on itself and taking $\bullet :=\mu $. We give two proofs of
		this fact. First, notice that for $a,a',m,m'\in A$, \eqref{equ:bimodule_algebra} can be written as
		\begin{equation*}
			(a,m)\aspr(a',m'):=(aa', (a+m)(a'+m')-aa')\;.
		\end{equation*}
		From this alternative formula for $\aspr$, it is clear by induction that
		\begin{equation*}
			(a_1,m_1)\aspr\cdots\aspr(a_n,m_n):=(a_1\dots a_n, (a_1+m_1)\dots(a_n+m_n)-a_1\dots a_n)\;,
		\end{equation*}
		where $a_1,\dots,a_n,m_1,\dots,m_n\in A$ and where $a_1\dots a_n$ and $(a_1+m_1)\dots(a_n+m_n)$ have the same
		parenthesizing as $(a_1,m_1)\aspr\cdots\aspr(a_n,m_n)$. It follows that if $\cR=0$ is a relation of degree $n$ of
		$\cC$, then for such elements
		\begin{equation*}
			\cR_{\aspr}((a_1,m_1),\cdots,(a_n,m_n))=(\cR(a_1,\dots,a_n),\cR((a_1+m_1),\cdots,(a_n+m_n))-\cR(a_1,\dots,a_n))\;,
		\end{equation*}
		so that $\cR_{\aspr}=0$ in $A\bowtie M$. Since this holds for any relation $\cR$ of $\cC$, this shows that $A\bowtie
		A\in\cC$, and hence that $A$ is an $A$-bimodule with these structures.
		
		We give a second proof, which uses the fact that the relations $\cR$ of $\cC$ are $n$-linear; we do so using a
		method and some notation which will be used several times in this paper. For $a\in A$, denote $\a0=(a,0)\in A\oplus
		A$ and $\ul a1=(0,a)\in A\oplus A$. According to \eqref{equ:bimodule_algebra}, the product of such elements is
		given by
		\begin{equation}\label{equ:A_bimodule}
			\a0\aspr \ul{a'}0=\ul{aa'}0, \qquad \a0\aspr \ul{a'}1= \a1\aspr \ul{a'}0= \a1\aspr \ul{a'}1=\ul{aa'}1\;,
		\end{equation}
		where $a,a'\in A$, and these products determine $\aspr$ completely since every element $(a,a')$ of $A\oplus A$ can
		be written as $(a,a')=\a0+\ul{a'}1$. We need to prove that $(A\oplus A,\aspr)\in\cC$, where $(a,x)\aspr
		(a',x')=(aa',ax'+xa'+xx')$, for all $a,a',x,x'\in A$. Let $\cR$ be an $n$-linear relation of $\cC$. We need to show
		that $\cR_{\aspr}=0$; by $n$-linearity, we only need to prove that $\cR_{\aspr}(u_1,\dots,u_n)=0$ if some of the
		elements $u_i$ are in $A_0:=A\oplus\{0\}$ and the others are in $A_1:=\{0\}\oplus A$. If all the elements $u_i$ are
		in $A_0$, say $u_i=\ul{a_i}0$, it is clear from the first formula in \eqref{equ:A_bimodule} that
		$\cR_{\aspr}(u_1,\dots,u_n)=\ul{\cR(a_1,\dots, a_n)}0=0$.  Similarly, if at least one of the elements $u_i$ is in
		$A_1$, say $u_i=\ul{a_i}1$, and all the other elements $u_j$ are either in $A_0$ or in
		$A_1$, say $u_j=\ul{a_j}0$ or $u_j=\ul{a_j}1$, then again it is clear from \eqref{equ:A_bimodule} that
		$\cR_{\aspr}(u_1,\dots,u_n)= \ul{\cR(a_1,\dots,a_n)}1=0$.
	\end{example}
	
	As was already pointed out in \cite[page 26]{schafer} $(A,\mu)$, equipped with the same left and right actions, is
	an $A$-bimodule. This can shown by an easy adaption of either of the arguments used in the above example; it is
	also a direct consequence of the following general proposition (using the above example):
	
	\begin{proposition}\label{warmup}
		Let $(A,\mu)\in \cC$ and let $(M,\pr)$ be an $A$-bimodule algebra. Then $M$ is an $A$-bimodule.
	\end{proposition}
	\begin{proof}
		The proof is another simple application of the method, explained in Example~\ref{exa:A_bimodule_algebra}.  Notice
		that the product $*$, defined in \eqref{equ:bimodule}, is graded, upon assigning a degree 0 to elements of $A$ and
		a degree~1 to elements of $M$ and that the grading is concentrated in degrees 0 and~1. Therefore, for any
		$n$-linear relation $\cR=0$, we have that $\cR_{*}(u_1,\dots,u_n)=0$ whenever at least two elements $u_i$ from $M$
		are substituted, and the other ones from $A$. When at most one $u_i$ belongs to $M$, then
		$\cR_{*}(u_1,\dots,u_n)=\cR_{*_\bullet}(u_1,\dots,u_n)$, as is clear by comparing \eqref{equ:bimodule} and
		\eqref{equ:bimodule_algebra}. Since by assumption the latter is zero, we may conclude that $\cR_*=0$ for any
		relation $\cR$ of $\cC$, so that $M$ is an $A$-bimodule.
	\end{proof}
	With a similar proof, one can show that, in the above definition of an extension $A\oplus_HM$, the condition that
	$M$ is an $A$-bimodule is a consequence of the other axioms.

	\section{NS-algebras}\label{sec:ns_algebras}
	In this section we introduce the notion of a $\cC$-NS-algebra for an arbitrary category~$\cC$ of binary algebras
	$(A,\mu)$ satisfying a given collection of multilinear relations $\cR_1=0$, \dots, $\cR_k=0$. Our definition
	generalizes the known notion in the particular cases of associative algebras, Lie algebras and Leibniz algebras. We
	will show how these cases are obtained from our definition and present several new examples. We also show that
	$\cC$-tridendriform algebras lead to $\cC$-NS-algebras, thereby generalizing a result by Uchino \cite{uchino_pre}
	and prove a general result (Theorem \ref{thm:general_ns}) allowing the construction of concrete $\cC$-NS-algebras
	from operators on $\cC$-algebras; three particular classes of such operators will be considered in the three
	subsequent sections.
	
	\begin{definition}\label{def:NS}
		An algebra $(A,\prec,\succ,\vee)$ is said to be a \emph{$\cC$-NS-algebra} if $(A\oplus A,\bt)\in\cC$, where~$\bt$
		is defined for $(a,x),(a',x')\in A\oplus A$ by
		\begin{equation}\label{eq:def_boxtimes}
			(a,x)\bt(a',x'):=(a\star a',a\succ x'+x\prec a')\;,\quad\hbox{where}\quad a\star a':= a\succ
			a'+a\prec a'+a\vee a'\;.
		\end{equation}
	\end{definition}
	
	Equivalently, $(A,\prec,\succ,\vee)$ is a {$\cC$-NS-algebra} if $(A,\star)\in\cC$, with $\star$ defined as in
	\eqref{eq:def_boxtimes}, and if $(A,\succ,\prec)$ is an $(A,\star)$-bimodule. We denote by $\catns$ the category
	whose objects are $\cC$-NS-algebras and whose morphisms are algebra homomorphisms.  Since $(A, \star )\in
	\mathcal{C}$, we have a functor from the category $\catns$ to~$\cC$: on objects it is given by
	$(A,\prec,\succ,\vee)\rightsquigarrow (A,\star)$, while it is identity on morphisms.
	
	A $\cC$-NS-algebra for which the products $\prec$ and $\succ$ are trivial is just an algebra of $\cC$. Also, a
	$\cC$-dendriform algebra (Definition \ref{def:dendri_gen}) is a $\cC$-NS-algebra for which $\vee$ is trivial.
	
	Since the relations are multilinear, the relations which any algebra in $\catns$ must satisfy are obtained by
	substituting, in every $n$-linear relation of $\cC$, $n-\ell$ elements from $A_0:= A\oplus\set0$ and $\ell$
	elements from $A_1:=\set0\oplus A$ and rewriting the result in terms of $\prec,\succ$ and $\vee$, using the
	following multiplication rules, in which $a$ and~$a'$ stand for arbitrary elements of $A$:
	\begin{equation}
		\a0\bt \ul{a'}0=\ul{a\star a'}0,\quad \a0\bt \ul{a'}1=\ul{a\succ a'}1,\quad \a1\bt \ul{a'}0=\ul{a\prec a'}1,\quad \a1\bt \ul{a'}1=(0,0). \label{eq:boxtimes}
	\end{equation}
	%
	%
	
	It is clear from these formulas that $(A\oplus A,\bt)$ is a graded algebra, when the elements of $A_i$ are assigned
	the degree $i$. Since the graduation is concentrated in degrees $0$ and $1$, it follows that if $\cR=0$ is any
	$n$-linear relation of $\cC$ and we substitute in $\cR_\bt$ at least two elements from~$A_1$, then we get
	zero. Therefore, any $n$-linear relation of $\cC$ leads to $n+1$ relations, some of which may coincide when $\cR$
	is invariant under a non-trivial permutation of its arguments. The first $n$ relations are obtained by taking one
	element in $A_1$ and $n-1$ elements in $A_0$. We call these $n$ relations ''of dendriform type'', because they are
	formally identical to the ones for a $\mathcal{C}$-dendriform algebra as in \cite{opv}, except that here $\star $
	means the sum of all three products $\prec $, $\succ $ and $\vee $, not the sum of $\prec $ and $\succ $ as in the
	case of $\mathcal{C}$-dendriform algebras; in particular, the operadic method of bisuccessors of \cite{baibellier}
	can also be used to determine these $n$ relations. The $(n+1)^{th}$-relation is obtained by taking all the $n$
	elements in $A_0$, i.e., no element in $A_1$, and by (\ref{eq:boxtimes}) this relation is equivalent to the
	condition that $(A, \star )\in \mathcal{C}$. It follows that by extending the operadic notion of \emph{arity
		splitting} in \cite[Section 2]{pei} so as to include the empty set our notion of NS-algebra can be extended to
	arbitrary operads.
	
	\begin{example}\label{exa:ns_assoc}
		As we already recalled in the introduction, the notion of an associative-NS-algebra was first considered by
		Leroux~\cite{leroux}, who defined them as algebras $(A,\prec,\succ,\vee)$, satisfying the 4 relations in
		\eqref{eq:assoc_NS}.
		%
		We show how these relations are obtained from the associativity of $\bt$. First, take $\a0,\b0$ in $A_0$
		and~$\c1$ in $A_1$. Then, by the associativity of $\bt$ and by (\ref{eq:boxtimes}),
		\begin{equation*}
			\ul{(a\star b)\succ c}1=(\a0\bt \b0)\bt \c1=\a0\bt (\b0\bt \c1)=\ul{a\succ(b\succ c)}1\;,
		\end{equation*}%
		so that $(a\star b)\succ c=a\succ(b\succ c)$, which is the third relation in \eqref{eq:assoc_NS}.  Taking $\a0,\c0$
		in $A_0$ and $\b1$ in $A_1$ (resp.\ $\b0,\c0$ in $A_0$ and $\a1$ in $A_1$) one obtains similarly the first and
		second relation in \eqref{eq:assoc_NS}.  For the fourth relation, one can take $(a\star b)\star c=a\star(b\star
		c)$. To see how the fourth relation in \eqref{eq:assoc_NS} is obtained from it, it suffices to notice that the sum
		of four relations in \eqref{eq:assoc_NS} is precisely the relation $(a\star b)\star c=a\star(b\star c)$, as we
		already pointed out. Therefore, the fourth relation in \eqref{eq:assoc_NS} is obtained by subtracting from the
		relation $(a\star b)\star c=a\star(b\star c)$ the three other relations in \eqref{eq:assoc_NS}. We will see in
		Remark \ref{forexamples} that there is an alternative way to obtain the last formula in this example, and in the
		examples which follow.
	\end{example}
	\begin{example}
		A \emph{(left) Leibniz algebra} is an algebra $(A,\mu)$ satisfying for all
		$a,b,c\in A$,
		\begin{equation}\label{eq:left_leibniz}
			a(bc)=(ab)c+b(ac)\;.
		\end{equation}
		So the relation defining Leibniz algebras is $\cR=(ab)c-a(bc)+b(ac)=0$. The first three relations defining
		Leibniz-NS-algebras are obtained as in Example \ref{exa:ns_assoc} by substituting in $\cR_\bt=0$ two elements from
		$A_0$ and one from $A_1$. Using (\ref{eq:boxtimes}), one finds
		\begin{align}
			a\succ(b\succ c)&=(a\star b)\succ c+b\succ(a\succ c)\;,\nonumber\\
			a\succ(b\prec c)&=(a\succ b)\prec c+b\prec(a\star c)\;, \label{eq:leibniz_ns}\\
			a\prec(b\star c)&=(a\prec b)\prec c+b\succ(a\prec c)\;.\nonumber
		\end{align}
		For the fourth relation, we take the difference of the relation $a\star(b\star c)=(a\star b)\star c+b\star(a\star
		c)$ and the above three relations. Since
		\begin{equation*}
			a\star(b\star c)-a\succ(b\succ c) -a\succ(b\prec c)-a\prec(b\star c)=a\succ(b\vee c)+a\vee(b\star c)\;,
		\end{equation*}
		and similarly for the other parenthesizing, it is given by
		\begin{equation}\label{eq:Leibniz_4}
			a\succ(b\vee c)+a\vee(b\star c)=(a\vee b)\prec c+(a\star b)\vee c+b\succ(a\vee c)+b\vee(a\star c)\;.
		\end{equation}
		We have hereby recovered that the relations defining Leibniz-NS-algebras are given by \eqref{eq:leibniz_ns} and
		\eqref{eq:Leibniz_4}, as in~\cite{dasleibniz}.
	\end{example}
	
	\begin{example}\label{exa:Lie}
		Since Lie algebras are anticommutative Leibniz algebras, we can use the previous example to easily determine the
		relations of Lie-NS-algebras. Anticommutativity is just an extra relation, which leads by the above method to the
		relations $a\succ b=-b\prec a$ and $a\vee b=-b\vee a$; in particular, $a\star b=-b\star a$. It is therefore natural
		to replace the products $\succ$ and $\prec$ by a single product, denoted $\times$; thus we set $a\succ b=a\times b$
		and $a\prec b=-b\times a$, so that $a\star b=a\times b-b\times a+a\vee b$. If we do this for the first equation in
		\eqref{eq:leibniz_ns}, we find
		\begin{equation}\label{eq:Lie_1}
			(a\star b)\times c=a\times(b\times c)-b\times(a\times c)\;;
		\end{equation}
		for the other equations in \eqref{eq:leibniz_ns} one finds the same relation, modulo a permutation of the
		variables.  For \eqref{eq:Leibniz_4} we find
		\begin{equation*}
			a\times(b\vee c)+a\vee(b\star c)=c\times(b\vee a)+(a\star b)\vee c+b\times(a\vee c)+b\vee(a\star c)\;.
		\end{equation*}
		Upon using the anticommutativity of $\vee$ and $\star$, the latter relation can be rewritten as
		\begin{equation}\label{eq:Lie_2}
			a\vee(b\star c)+b\vee(c\star a)+c\vee(a\star b)+a\times(b\vee c)+b\times(c\vee a)+c\times(a\vee b)=0\;.
		\end{equation}
		It follows that \eqref{eq:Lie_1}, \eqref{eq:Lie_2} and the anticommutativity of $\vee$, namely the relation $a\vee
		b=-b\vee a$, are the relations of a Lie-NS-algebra. These were first given in \cite{dastwistedLie}.
	\end{example}
	\begin{example}\label{exa:NAP}
		In order to give a first new example, let us consider NAP algebras, which are defined by the relation
		$a(bc)=b(ac)$. By the above method, we find that NAP-NS-algebras are algebras satisfying the following three
		relations:
		\begin{align*}
			a\succ(b\succ c)&=b\succ (a\succ c)\;,\\
			a\succ(b\prec c)&=b\prec (a\star c)\;,\\
			a\succ(b\vee c)+a\vee(b\star c)&=b\succ(a\vee c)+b\vee(a\star c)\;.
		\end{align*}
		There are only three relations because the NAP relation $a(bc)=b(ac)$ is invariant under the transposition which
		exchanges $a$ and $b$.
	\end{example}
	
	\begin{example}\label{exa:jordan}
		We next consider the example of Jordan algebras. Recall that a Jordan algebra is a commutative algebra $A$,
		satisfying the Jordan identity $(ab)(aa)=a(b(aa))$. Suppose first that our base ring $R$ is a field whose
		characteristic is different from $2$ and $3$.  Then, according to \cite{schafer}, the Jordan identity is equivalent
		to its linearized form, which is given by
		\begin{equation}\label{eq:jordan_linearized}
			(ad)(bc)+(bd)(ac)+(cd)(ab)=a(d(bc))+b(d(ac))+c(d(ab))\;,
		\end{equation}
		and we can apply the above method to it and to the commutativity equation $ab=ba$. For future comparison, we give
		the result. First, we have of course the condition that $(A,\star)\in \cC$, where $\cC$ stands for the category of
		Jordan algebras,
		\begin{equation}\label{eq:jordan_1}
			(a\star d)\star (b\star c)+(b\star d)\star (a\star c)+(c\star d)\star (a\star b)=a\star (d\star (b\star
			c))+b\star (d\star (a\star c))+c\star (d\star (a\star b))\;.
		\end{equation}
		Second, commutativity leads similarly to the anticommutativity in Example \ref{exa:Lie} to $a\succ b=b\prec a$ and
		$a\vee b=b\vee a$ and we set again $a\times b:=a\succ b=b\prec a$, so $\star$ is commutative and $a\star b$ becomes
		$a\times b+b\times a+a\vee b$. Since \eqref{eq:jordan_linearized} is symmetric in $a,b,c$, there will only be two
		extra equations which come from the Jordan identity and they can be written as the following double equality:
		\begin{align*}
			&(b\star c)\times(d\times a)+(b\star d)\times(c\times a)+(c\star d)\times(b\times a)\\
			&\qquad=(d\star(b\star c))\times a+b\times(d\times(c\times a))+c\times(d\times(b\times a))\\
			&\qquad=d\times((b\star c)\times a)+b\times((c\star d)\times a)+c\times((b\star d)\times a))\;.
		\end{align*}
		Suppose now that $R$ is any ring, as before, and notice that our Definition \ref{def:NS} makes sense even when the
		relations are not multilinear. We show how to determine the relations that every Jordan-NS-algebra, in this
		generalized sense, must satisfy. Notice that the methods of bisuccessors and splittings in \cite{baibellier,pei}
		only work for multilinear relations, so that in the operadic context the Jordan identity is only considered in its
		linearized form. Of course, we can deal with the commutativity relation as before and only need to consider the
		Jordan identity $(ab)(aa)=a(b(aa))$, which amounts to write out the following identity:
		\begin{equation}\label{eq:jordan_NS}
			((a,x)\bt(b,y))\bt((a,x)\bt(a,x))=(a,x)\bt((b,y)\bt((a,x)\bt (a,x)))\;,
		\end{equation}
		where $a,b,x,y$ are arbitrary elements of $A$. Equality of the first components in \eqref{eq:jordan_NS} yields
		the relation
		\begin{equation}\label{eq:jordan_general}
			(a\star b)\star(a\star a)=a\star(b\star(a\star a))\;,
		\end{equation}
		which just states that $(A,\star)\in \cC$. Equality of the second components yields
		\begin{align}\label{eq:jordan_big}
			&(a\star a)\times(b\times x)+(a\star a)\times(a\times y)+2(a\star b)\times(a\times x)\nonumber\\
			&\qquad=(b\star(a\star a))\times x+a\times((a\star a)\times y)+2a\times(b\times(a\times x))\;.
		\end{align}
		In it, take $b=x=0$, to find the following simple relation:
		\begin{equation}\label{eq:jordan_simple}
			(a\star a)\times(a\times y)=a\times((a\star a)\times y)\;.
		\end{equation}
		It can be used to simplify \eqref{eq:jordan_big} to obtain the relation
		\begin{equation}\label{eq:jordan_medium}
			(a\star a)\times(b\times x)+2(a\star b)\times(a\times x)=(b\star(a\star a))\times x+2a\times(b\times(a\times x))\;.
		\end{equation}
		In conclusion, a Jordan-NS-algebra algebra is an algebra $(A,\times,\vee)$, satisfying the relations
		\eqref{eq:jordan_general}, \eqref{eq:jordan_simple} and \eqref{eq:jordan_medium}, where $a\star b=a\times b+b\times
		a+a\vee b$, with $\vee$ being commutative.
	\end{example}

	\begin{example}
		As a final example, we consider Poisson algebras, so as to explain the minor adaptions when the category of
		algebras $\cC$ consists of algebras with several binary products. Let us first recall that a Poisson algebra $A$
		comes equipped with two binary products $a\times b\mapsto ab$ and $a\times b\mapsto\pb{a,b}$, where the first
		product is commutative and associative, the second product is a Lie bracket, and the two structures are compatible
		in the sense that
		\begin{equation}\label{eq:Leibniz}
			\pb{ab,c}=a\pb{b,c}+b\pb{a,c}\;,
		\end{equation}
		a property which is often called the \emph{Leibniz identity}, but which should not be confused with the defining
		property \eqref{eq:left_leibniz} of a (left) Leibniz algebra, which involves only one product.
		
		When dealing with several products on the algebra $A\in\cC$, one needs to generalize Definition \ref{def:NS} and
		introduce for every given product on $A$ corresponding new products $\succ,\prec$ and $\vee$ on $A$ and a
		corresponding product $\boxtimes$ on $A\oplus A$; then one demands that $A\oplus A$, equipped with all these
		products also belong to $\cC$. Rather than writing down the general definition, let us spell out in some detail the
		case of a Poisson algebra, where there are only two products, which are moreover commutative, respectively
		anticommutative; the latter properties allow us as in Examples \ref{exa:ns_assoc} and \ref{exa:Lie} to replace each
		pair of products $(\succ,\prec)$ by a single product. Explicitly, this means that an algebra
		$(A,*,\vee,\circ,\wedge)$ is said to be a \emph{Poisson-NS-algebra} if $(A\oplus A,\odot,\GB)$ is a Poisson
		algebra, where $\odot$ and $\GB$ are defined~by
		\begin{align*}
			(a,x)\odot (b,y):=(a*b+b*a+a\vee b,a*y+b*x)\;,\\
			\gb{(a,x), (b,y)}:=(a\circ b-b\circ a+a\wedge b,a\circ y-b\circ x)\;.
		\end{align*}%
		In order to simplify some of the formulas below, we will use the following shorthands:
		\begin{equation}\label{eq:tsar}
			a\star b:=a*b+b*a+a\vee b\;,\qquad\qquad   a\tsar b:=a\circ b-b\circ a+a\wedge b\;.
		\end{equation}
		We have already analyzed commutativity, associativity and being a Lie bracket in Examples \ref{exa:ns_assoc},
		\ref{exa:Lie} and \ref{exa:jordan}. We therefore know that $\vee$ must be commutative, $\wedge$ anticommutative,
		and that we must already have the following relations for $(A,*,\vee,\circ,\wedge)$ to be a Poisson-NS-algebra,
		where the first three correspond to associativity and the last two to the Jacobi identity:
		\begin{align}\label{eq:poisson_1}
			a*(b*c)&=b*(a*c)=(a\star b)*c\;,\\
			(a\star b)\star c&=a\star(b\star c)\;,\\
			(a\tsar b)\circ c&=a\circ(b\circ c)-b\circ(a\circ c)\;,\\
			a\tsar(b\tsar c)&=(a\tsar b)\tsar c+b\tsar(a\tsar c)\,.
		\end{align}
		We need to add to these relations the relations that come from the Leibniz identity. Due to the symmetry between
		$a$ and $b$ in \eqref{eq:Leibniz} we only get three relations which are again obtained by substituting at most one
		element from $A_1$ and all the other elements from $A_0$ in the Leibniz identity on $A\oplus A$,
		\begin{equation*}
			\gb{(a,x)\odot(b,y),(c,z)}=(a,x)\odot\gb{(b,y),(c,z)}+(b,y)\odot\lb{(a,x)\odot(c,z)}\;.
		\end{equation*}
		It leads to the following relations:
		\begin{align}
			(a\star b)\tsar c&=a\star(b\tsar c)+b\star(a\tsar c)\;,\\
			(a\star b)\circ c&=a*(b\circ c)+b*(a\circ c)\;,\\
			(a\tsar b)*c&=a* (b\circ c)-b\circ(a*c)\;.\label{eq:poisson_2}
		\end{align}
		It follows that $(A,*,\vee,\circ,\wedge)$ is a Poisson-NS-algebra when it verifies \eqref{eq:poisson_1} --
		\eqref{eq:poisson_2}, $\vee$ is commutative and $\wedge$ is anticommutative, where we recall that the abbreviations
		$\star$ and $\tsar$ have been defined in \eqref{eq:tsar}. Recently, an equivalent set of relations has been
		proposed in~\cite{daspoisson}, but without reference to the bimodule property.
	\end{example}
	
	In order to give another class of examples, we recall from \cite{opv} the definition of a $\cC$-tridendriform
	algebra and show that such an algebra is a $\cC$-NS-algebra. In the associative case, this was already pointed out
	by Uchino in a preprint version of \cite{uchino}.
	
	\begin{definition}\label{def:tridendri}
		An algebra $(A,\prec,\succ,\wee)$ is called a \emph{$\cC$-tridendriform algebra} if \hbox{$(A\oplus
			A,\boxtimes)\in\cC$}, where $\boxtimes$ is defined for all $(a,x),(a',x')\in A\oplus A$ by
		\begin{equation}\label{eq:def_boxtimes_trid}
			(a,x)\boxtimes(a',x'):=(a\star a',a\succ x'+x\prec a'+x\wee x')\;,\quad\hbox{where}\quad a\star a':= a\succ
			a'+a\prec a'+a\wee a'\;.
		\end{equation}
	\end{definition}
	
	Similarly to the case of $\cC$-NS-algebras, this is equivalent to saying: $(A,\star)\in\cC$ and
	$(A,\wee,\succ,\prec)$ is an $(A,\star)$-bimodule algebra (this characterization extends to arbitrary $\cC$ the one
	for the category of associative algebras given in Proposition 6.12 in \cite{BaiGuoNew}).
	
	\begin{proposition}\label{prp:tridendri_is_NS}
		Let $(A,\prec,\succ,\wee)$ be a $\cC$-tridendriform algebra. Then $(A,\prec,\succ,\wee)$ is a $\cC$-NS-algebra.
	\end{proposition}
	
	\begin{proof}
		Suppose that $(A,\prec,\succ,\wee)$ is a $\cC$-tridendriform algebra. Then $(A,\wee,\succ,\prec)$ is an
		$(A,\star)$-bimodule algebra, hence $(A,\succ,\prec)$ is an $(A,\star)$-bimodule (by Proposition
		\ref{warmup}). Since $(A,\star)\in\cC$, this shows that $(A,\prec,\succ,\wee)$ is a $\cC$-NS-algebra, which proves
		the proposition.
	\end{proof}
	%
	%
	It follows directly from the proposition that every post-Lie algebra (see \cite{nonabelian} and \cite{valette}),
	which in our language is a Lie tridendriform algebra, is a Lie-NS-algebra.
	
	We now state and prove that a general class of operators leads to $\cC$-NS-algebras.
	\begin{theorem}\label{thm:general_ns}
		Let $(A,\mu)\in\cC$ and let $M$ be an $A$-bimodule. Let $\op:M\to A$ and $\alpha:M\otimes M\to M$ be linear
		maps. Define new products on $M$ by setting, for all $m,m'\in M$,
		\begin{gather*}
			m\succ m':=\op(m)\cdot m'\;,\quad m\prec m':= m\cdot \op(m')\;,\quad m\vee
			m':=\alpha(m\otimes m')\;,\\ m\star m':= m\succ m'+m\prec m'+m\vee m'\;.
		\end{gather*}
		If $\op:(M,\star)\to(A,\mu)$ is an algebra homomorphism and $(M,\star)\in\cC$, then $(M,\prec,\succ,\vee)$ is
		a~\hbox{$\cC$-NS-algebra.}
	\end{theorem}
	\begin{proof}
		We need to prove that, under the hypothesis of the theorem, $(M\oplus M,\bt)\in \cC$, where $(m, x)\bt
		(m', x')=(m\star m', m\succ x'+x\prec m')$, for all $m,m',x,x'\in M$. We denote by $M_0$ and~$M_1$ the
		submodules of $M\oplus M$ defined by $M_0=M\oplus \{0\}$, $M_1=\{0\}\oplus M$, and for $m\in M$ we denote
		$\underline{m}_0=(m, 0)$ and $\underline{m}_1=(0, m)$. When we work inside $(A\oplus M, *)$, we denote $a=(a, 0)$
		and $m=(0, m)$, for $a\in A$ and $m\in M$. With this notation, the products $*$ of $A\oplus M$ and $\bt$ of
		$M\oplus M$
		are completely determined by the following list (for all $a,a'\in A$ and $m, m'\in M$):
		\begin{align*}
			a*a'&=aa'\;,&a*m'&=a\cdot m'\;,&      m*a'&=m\cdot a'\;,&       m*m'&=0\;,\\
			\m0\bt\mp0&=\ul{m\star m'}0\;,&\m0\bt\mp1&=\ul{\op(m)\cdot m'}1\;,
			&\m1\bt\mp0&=\ul{m\cdot\op(m')}1\;,&\m1\bt\mp1&=0\;.
		\end{align*}
		Let~$\cR$ be an $n$-linear relation of $\cC$. We need to prove that $\cR_\bt(u_1,\dots,u_n)=0$ for all $u_1,\dots,
		u_n\in M\oplus M$. By $n$-linearity, it is enough to prove this when some of the elements $u_i$ are in $M_0$ and
		the others are in $M_1$. By the above list of products it is clear that the product $\bt$ is graded, upon
		assigning a degree 0 to elements of $M_0$ and a degree~1 to elements of $M_1$. Since the grading is concentrated in
		degrees 0 and~1, $\cR_\bt(u_1,\dots,u_n)=0$ if at least two of the elements $u_i$ are in $M_1$. If all the elements
		$u_i$ are in $M_0$, say $u_i=\ul{m_i}0$, then $\cR_\bt(u_1,\dots,u_n)=\ul{\cR_\star(m_1,\dots,m_n)}0=0$, the second
		equality being a consequence of the assumption that $(M,\star)\in\cC$.  It remains to consider the case in which
		one element is in $M_1$ and the other $n-1$ elements are in $M_0$. Consider a monomial
		$X=m_1m_2\dots m_n$ of length $n$ with any parenthesizing and denote for $1\leqslant\ell\leqs n$:
		\begin{align*}
			X_*^0&:=\op(m_1)*\op(m_2)*\cdots*\op(m_n)=\op(m_1)\op(m_2)\dots\op(m_n)=\op(m_1\star m_2\star\cdots\star m_n)\;,\\
			X_*^\ell&:=\op(m_1)*\cdots*\op(m_{\ell-1})*m_\ell*\op(m_{\ell+1})*\cdots*\op(m_n)\;,\\
			X_\bt^0&:=\ul{m_1}0\bt\ul{m_2}0\bt\cdots\bt\ul{m_n}0=\ul{m_1\star m_2\star\cdots\star m_n}0=\ul{X_\star}0\;,\\
			X_\bt^\ell&:=\ul{m_1}0\bt\ul{m_2}0\bt\cdots\bt\ul{m_{\ell-1}}0\bt\ul{m_\ell}1\bt\ul{m_{\ell+1}}0\bt\cdots\bt\ul{m_n}0\;.
		\end{align*}
		In the first line we have used the assumption that $\op:(M,\star)\to(A,\mu)$ is an algebra homomorphism. We show by
		induction on $n$ that
		\begin{equation}\label{equ:bt_vs_*}
			X_\bt^\ell=\ul{X_*^\ell}1\;,\quad\hbox{ for } \quad 1\leqslant\ell\leqs n\;.
		\end{equation}
		Notice that $X_*^\ell\in\set0\oplus M\simeq M$, so that the right hand side of \eqref{equ:bt_vs_*} makes sense.
		For $n=2$ and $\ell=1$, $X=m_1m_2$ so that
		\begin{equation*}
			X_\bt^1=\ul{m_1}1\bt\ul{m_2}0=\ul{m_1\cdot\op(m_2)}1\;,\quad\hbox{and}\quad
			\ul{X_*^1}1=\ul{m_1*\op(m_2)}1=\ul{m_1\cdot\op(m_2)}1\;,
		\end{equation*}
		as was to be shown; for $\ell=2$ the proof is similar. Assume now that \eqref{equ:bt_vs_*} holds whenever the
		length of $X$ is smaller than $n$. We can write $X$ uniquely as $X=YZ$, where~$Y$ and~$Z$ inherit  parenthesizings
		from $X$. Let us assume first that $\ell$ is at most the length of $Y$; using the above notations also for $Y$ and
		$Z$ we then have $X_\bt^\ell=Y_\bt^\ell\bt Z_\bt^0$. Using the induction hypothesis, the formulas for $\bt$ and the
		fact that $\op$ is a algebra homomorphism, we find
		\begin{align*}
			X_\bt^\ell&=Y_\bt^\ell\bt Z_\bt^0\overset{\eqref{equ:bt_vs_*}}=\ul{Y_*^\ell}1\bt\ul{Z_\star}0
			\overset{\bt}=\ul{{Y_*^\ell}\cdot\op(Z_\star)}1\overset{\op}=\ul{{Y_*^\ell}\cdot Z_*^0}1
			=\ul{{Y_*^\ell}* Z_*^0}1=\ul{X_*^\ell}1\;,
		\end{align*}
		as was to be shown. By symmetry, \eqref{equ:bt_vs_*} also holds when $\ell$ is larger than the length of~$Y$.
		
		Since any $n$-linear relation $\cR$ of $\cC$ is a linear combination of such parenthesized monomials $X$, we get
		from \eqref{equ:bt_vs_*}, using the analogous notations for $\cR$, that $\cR_\bt^\ell=\ul{\cR_*^\ell}1=0$ for
		any $1\leqs\ell\leqs n$, where we have used in the last step that $\cR_*=0$, which follows from the fact that
		$(A\oplus M,*)\in\cC$ (since $M$ is an $A$-bimodule). This proves that $\cR_\bt=0$, as was to be shown.
	\end{proof}

	\goodbreak
	
	\section{Nijenhuis operators}\label{sec3}
	
	Nijenhuis operators have already been considered in the literature in the case of associative, Lie, pre-Lie and
	Leibniz algebras. A key property is that a Nijenhuis operator on such a type of algebra produces a new algebra of
	the same type. For associative, Lie, pre-Lie and Leibniz algebras, this was proved in \cite{carinena},
	\cite{yvette2}, \cite{NijpreLie} and respectively \cite{dasleibniz}.  A first result of this section is Theorem
	\ref{mainNijprop}, which states that this property generalizes to arbitrary binary algebras, where the definition
	of a Nijenhuis operator is on these algebras formally the same as in the cited cases (Definition
	\ref{def:Nijenhuis}). Our proof uses a technical lemma on powers of Nijenhuis operators on arbitrary binary
	algebras, which was already known in the pre-Lie case \cite{NijpreLie}. The second result, which we obtain as a
	corollary of the theorem and our general Theorem \ref{thm:general_ns}, states that Nijenhuis operators on arbitrary
	binary algebras lead to NS-algebras. Again this generalizes a result which was previously only known for the
	particular cases of associative, Lie and Leibniz algebras, see respectively \cite{leroux}, \cite{dastwistedLie} and
	\cite{dasleibniz}. We also give a few examples of Nijenhuis operators, which generalize known ones.  Throughout the
	section we denote by $\cC$ the category of all binary algebras $(A,\mu)$ satisfying a given collection of
	multilinear relations $\cR_1=0,\dots,\cR_k=0$.
	\begin{definition}\label{def:Nijenhuis}
		Let $(A,\mu)\in\cC$. A linear map $\fN:A\rightarrow A$ is called a \emph{Nijenhuis operator} (for $A$) if, for
		all $a,a'\in A$,
		\begin{equation}\label{Nij}
			\fN(a)\fN(a')=\fN(\fN(a)a'+a\fN(a')-\fN(aa'))\;.
		\end{equation}
	\end{definition}
	We denote by $\catnij$ the category whose objects are pairs $(A,\fN)$, where $A=(A,\mu)$ is an algebra in $\cC$ and
	$\fN$ is a Nijenhuis operator for $A$; a \emph{morphism} between two such pairs $(A,\fN)$ and $(A',\fN')$ is an
	algebra homomorphism $f:A\to A'$, satisfying $f\circ\fN=\fN'\circ f$.
	
	\begin{example}
		For any $A\in\cC$, it is clear that $\Id_A$ is a Nijenhuis operator for $A$. More generally, when $\fN$ is a
		Nijenhuis operator for $A$, then any linear combination of $\fN$ and $\Id_A$ is a Nijenhuis operator for~$A$. In
		fact, by using Lemma \ref{pairsij} below, one can easily prove that if $\fN$ is a Nijenhuis operator for $A$ and
		$P(z)=\displaystyle\sum _{i=0}^mc_iz^i$ is a polynomial, then the operator $P(\fN)$ is also a Nijenhuis operator
		for $A$. This generalizes to arbitrary $\cC$ a result which was already shown for the particular cases of
		associative algebras and pre-Lie algebras in \cite{carinena}, respectively in \cite{NijpreLie}.
	\end{example}
	\begin{example}
		Let $(A, \mu )\in \mathcal{C}$ and assume that $A$ is a \emph{twilled algebra}, i.e.\ it is equipped with a
		decomposition $A=A_1\oplus A_2$, where $A_1$ and $A_2$ are subalgebras of $A$. Then, with the same proof as in the
		associative case in \cite{carinena}, one can see that, if we denote by $P_i$ the projection onto $A_i$, for $i=1,
		2$, then any linear combination $\fN=\lambda _1P_1+\lambda _2P_2$ is a Nijenhuis operator for $A$.  A class of
		examples of twilled algebras is provided by semidirect sums $A\bowtie M$ as in Section~\ref{sec1}. As a concrete
		example, let $(A, \mu )\in \mathcal{C}$; since, by Example~\ref{exa:A_bimodule_algebra}, $A$ is an $A$-bimodule
		algebra, it follows that $(A\oplus A, *_\mu)\in \mathcal{C}$ is a twilled algebra, where $(a,x)*_\mu (a',x')=(aa',
		ax'+xa'+xx')$, for all $a, a', x, x'\in A$.
	\end{example}
	\begin{theorem}\label{mainNijprop}
		Let $(A,\mu)\in \cC$ and $\fN:A\rightarrow A$ a Nijenhuis operator.  Define a new product $\star$ on~$A$ by
		setting $ a\star a':=\fN(a)a'+a\fN(a')-\fN(aa'),$ for all $a,a'\in A$.  Then $(A,\star)\in \cC$.
	\end{theorem}
	
	The theorem leads to a functor $\catnij\to\cC$, which is given on objects by $(A,\fN)\rightsquigarrow (A,\star)$,
	and is identity on morphisms. For the proof of Theorem \ref{mainNijprop}, we will use the following lemma:
	
	\begin{lemma}\label{pairsij}
		Suppose that $\fN$ is a Nijenhuis operator on $A\in\cC$ and let $a,b\in A$. For any $i,j\in\bbN^*$,
		\begin{equation}\label{eq:opmn}
			\fN^i(a)\fN^j(b)=\fN^j(\fN^i(a)b)+\fN^i(a\fN^j(b))-\fN^{i+j}(ab)\;.
		\end{equation}
	\end{lemma}
	\begin{proof}
		We use induction on the poset $(\bbN^*\times\bbN^*,\leqs)$, where $(i,j)\leqs(i',j')$ if and only if $i\leqs i'$
		and $j\leqs j'$. For $(i,j)=(1,1)$, Equation \eqref{eq:opmn} is just \eqref{Nij}. Let $(i,j)>(1,1)$ and suppose
		that \eqref{eq:opmn} is true for all exponents smaller than $(i,j)$. Suppose first that $i=1$, so that
		$j>1$. Using \eqref{Nij} and \eqref{eq:opmn}, in that order, we get
		\begin{align*}
			\fN(a)\fN^{j}(b)&\overset{\eqref{Nij}}=\fN(\fN(a)\fN^{j-1}(b))+\fN(a\fN^{j}(b))-\fN^2(a\fN^{j-1}(b))\\
			&\overset{\eqref{eq:opmn}}=\fN^{j}(\fN(a)b)+\fN(a\fN^{j}(b))-\fN^{j+1}(ab)\;,
		\end{align*}
		which shows that \eqref{eq:opmn} holds for exponents of the type $(1, j)$, and similarly for those of type $(i,
		1)$.  It remains to be shown that \eqref{eq:opmn} holds for $(i,j)$, with $i,j\geqs2$ when it holds for all
		$(i',j')<(i,j)$. This is done as above by using first \eqref{Nij} and then \eqref{eq:opmn} (three times),
		\begin{align*}
			\fN^i(a)\fN^j(b)
			&\overset{\eqref{Nij}}=\fN(\fN^i(a)\fN^{j-1}(b))+\fN(\fN^{i-1}(a)\fN^j(b))-\fN^2(\fN^{i-1}(a)\fN^{j-1}(b))\\
			&\overset{\eqref{eq:opmn}}=\fN^j(\fN^i(a)b)+\fN^{j+1}(\fN^{i-1}(a)b)-\fN^{j+1}(\fN^{i-1}(a)b)\\
			&\qquad+\fN^{i+1}(a\fN^{j-1}(b))+\fN^i(a(\fN^j(b))-\fN^{i+1}(a\fN^{j-1}(b))\\
			&\qquad-\fN^{i+j}(ab)-\fN^{i+j}(ab)+\fN^{i+j}(ab)\\
			&=\fN^j(\fN^i(a)b)+\fN^i(a\fN^j(b))-\fN^{i+j}(ab)\;,
		\end{align*}
		as was to be shown.
	\end{proof}
	\begin{proof}(of Theorem \ref{mainNijprop})
		Let $X=a_1a_2\dots a_n$ be an $n$-linear monomial in $A\in\cC$, with some parenthesizing. For $0\leqs j\leqs n$,
		let us denote by $X\pow j$ the sum of all monomials obtained by applying $\fN$ to $j$ factors of $X$, and this in
		all $n\choose j$ possible ways (with the same parenthesizing); by definition, $X\pow j=0$ for $j>n$. The same
		notation will be used for $n$-linear relations, which are just linear combinations of such monomials $X$. We will
		show that
		\begin{equation}\label{eq:X_star}
			X_\star=\sum_{\substack{i+j=n-1\\ i,j\geqs0}}(-\fN)^i\(X\pow j\)\;,
		\end{equation}
		which leads at once to the proof of the theorem. Indeed, let $\cR=0$ be an $n$-linear relation of $\cC$ and notice
		that $\cR\pow j=0$ for any $j$. Then \eqref{eq:X_star} implies that
		\begin{equation}\label{eq:R_star}
			\cR_\star=\sum_{\substack{i+j=n-1\\ i,j\geqs0}}(-\fN)^i\(\cR\pow j\)=0\;.
		\end{equation}
		It follows that $\cR_\star=0$ for any multilinear relation $\cR=0$ of $\cC$, and hence that $(A,\star)\in\cC$, as
		was to be shown. We still need to show \eqref{eq:X_star}, which we do by induction on $n\geqs2$. Notice that $X\pow
		n=\fN(X_\star)$, as follows from an easy induction on $n\geqs2$, the case of $n=2$ being just the property
		\eqref{Nij}. It follows that \eqref{eq:X_star} implies that
		\begin{equation}\label{eq:X^n}
			X\pow n=-\sum_{i=1}^n(-\fN)^i(X\pow{n-i})\;,
		\end{equation}
		a formula which we will also use in the inductive proof. When $n=2$, $X=a_1a_2$ and we have
		\begin{equation}
			X_\star=a_1\star a_2=\fN(a_1)a_2+a_1\fN(a_2)-\fN(a_1a_2)=X\pow1-\fN(X\pow 0)\;,
		\end{equation}
		so that \eqref{eq:X_star} is valid for $n=2$ (for $n=1$ the formula is also valid, trivially). Let $X$ be an
		$n$-linear monomial and suppose that \eqref{eq:X_star} is valid for any $k$-linear monomial, with $k<n$. We can
		write $X=YZ$, where $Y$ and $Z$ are monomials of length $s$ and $t$ respectively; the decomposition $X=YZ$ is
		uniquely determined by the parenthesizing, and $n=s+t$. Then
		\begin{equation*}
			X_\star=Y_\star\star Z_\star=\fN(Y_\star)Z_\star+Y_\star\fN(Z_\star)-\fN(Y_\star Z_\star)\;.
		\end{equation*}
		In order to avoid many signs in the proof, we will write here $\po$ for $-\fN$; notice that $\po$ is also a
		Nijenhuis operator, hence also satisfies \eqref{eq:opmn}, with the same signs. Using the induction hypothesis four
		times, together with $\fN(Y_\star)=Y\pow s$ and $\fN(Z_\star)=Z\pow t$,
		\begin{align*}
			X_\star&\overset{\eqref{eq:X_star}}=Y\pow s\(Z\pow {t-1}+\sum_{j=2}^t\po^{j-1}(Z\pow{t-j})\)+
			\(Y\pow {s-1}+\sum_{i=2}^s\po^{i-1}(Y\pow{s-i})\)Z\pow t\\
			&\qquad+\sum_{i=1}^s\sum_{j=1}^t\po\(\po^{i-1}(Y\pow{s-i})\po^{j-1}(Z\pow{t-j})\)\;.
		\end{align*}
		Let us call the three terms of this expression I, II and III, in that order. Using \eqref{eq:X^n} (twice)
		and~\eqref{eq:opmn}, I can be written as
		\begin{align*}
			\hbox{I}
			&\overset{\eqref{eq:X^n}}=Y\pow sZ\pow {t-1}-\sum_{i=1}^s\sum_{j=2}^t\po^i(Y\pow{s-i})\po^{j-1}(Z\pow{t-j})\\
			&\overset{\eqref{eq:opmn}}=Y\pow sZ\pow {t-1}-\sum_{i=1}^s\sum_{j=2}^t\po^{j-1}\(\po^i(Y\pow{s-i})Z\pow{t-j}\)
			-\sum_{i=1}^s\sum_{j=2}^t\po^i\(Y\pow{s-i}\po^{j-1}(Z\pow{t-j})\)\\
			&\qquad+\sum_{i=1}^s\sum_{j=2}^t\po^{i+j-1}\(Y\pow{s-i}Z\pow{t-j}\)\\
			&\overset{\eqref{eq:X^n}}=Y\pow sZ\pow {t-1}+\sum_{j=2}^t\po^{j-1}\(Y\pow{s}Z\pow{t-j}\)
			-\sum_{i=1}^s\sum_{j=2}^t\po^i\(Y\pow{s-i}\po^{j-1}(Z\pow{t-j})\)\\
			&\qquad+\sum_{i=1}^s\sum_{j=2}^t\po^{i+j-1}\(Y\pow{s-i}Z\pow{t-j}\)\;,
		\end{align*}
		and, by symmetry,
		\begin{align*}
			\hbox{II}&=Y\pow{s-1}Z\pow t+\sum_{i=2}^s\po^{i-1}\(Y\pow{s-i}Z\pow{t}\)
			-\sum_{i=2}^s\sum_{j=1}^t\po^j\(\po^{i-1}(Y\pow{s-i})Z\pow{t-j}\)\\
			&\qquad+\sum_{i=2}^s\sum_{j=1}^t\po^{i+j-1}\(Y\pow{s-i}Z\pow{t-j}\)\;.
		\end{align*}
		In order to rewrite III, we only use \eqref{eq:opmn}:
		\begin{align*}
			\hbox{III}&=\po(Y\pow{s-1}Z\pow{t-1})+
			\raisebox{-4.5ex}{$\stackrel{\displaystyle\sum_{i=1}^s\sum_{j=1}^t}{\scriptscriptstyle i+j>2}$}
			\po^j\(\po^{i-1}(Y\pow{s-i})Z\pow{t-j}\)
			+\raisebox{-4.5ex}{$\stackrel{\displaystyle\sum_{i=1}^s\sum_{j=1}^t}{\scriptscriptstyle i+j>2}$}
			\po^{i}\(Y\pow{s-i}\po^{j-1}(Z\pow{t-j})\)\\
			&\qquad-\raisebox{-4.5ex}{$\stackrel{\displaystyle\sum_{i=1}^s\sum_{j=1}^t}{\scriptscriptstyle i+j>2}$}
			\po^{i+j-1}\(Y\pow{s-i}Z\pow{t-j}\)\;.
		\end{align*}
		The main thing to notice now is that in the sum I$+$II$+$III \emph{all} terms which are \emph{not} of the form
		$\po^i(Y\pow jZ\pow k)$ for some $i,j,k$ cancel. It suffices then to collect all remaining terms, to find
		\begin{align*}
			X_\star&=\hbox{I}+\hbox{II}+\hbox{III}\\
			&=Y\pow sZ\pow {t-1}+Y\pow{s-1}Z\pow t+\po(Y\pow{s-1}Z\pow{t-1})+
			\sum_{j=2}^t\po^j(Y\pow{s-1}Z\pow{t-j})+\sum_{i=2}^s\po^i(Y\pow{s-i}Z\pow{t-1})\\
			&\qquad+\sum_{j=2}^t\po^{j-1}\(Y\pow{s}Z\pow{t-j}\)+\sum_{i=2}^s\po^{i-1}\(Y\pow{s-i}Z\pow{t}\)+
			{\sum_{i=2}^s\sum_{j=2}^t}\po^{i+j-1}\(Y\pow{s-i}Z\pow{t-j}\)\\
			&=\sum_{i=2}^s\sum_{j=0}^t\po^{i+j-1}\(Y\pow{s-i}Z\pow{t-j}\)+
			\sum_{i=0}^1\sum_{j=2}^t\po^{i+j-1}\(Y\pow{s-i}Z\pow{t-j}\)\\
			&\qquad+Y\pow sZ\pow {t-1}+Y\pow{s-1}Z\pow t+\po(Y\pow{s-1}Z\pow{t-1})\\
			&=\sum_{\substack{i+j+k=s+t-1\\i,j,k\geqs0}}\po^i(Y\pow jZ\pow k)\;,
		\end{align*}
		with the same parenthesizing. Since, for any $\ell\leqs n$, $\displaystyle X\pow\ell=\sum_{j+k=\ell}Y\pow jZ\pow
		k$, it follows that
		\begin{equation*}
			X_\star=\sum_{\substack{i+\ell=n-1\\i,\ell\geqs0}}\po^i(X\pow \ell)=
			\sum_{\substack{i+j=n-1\\i,j\geqs0}}(-\fN)^i(X\pow j)\;,
		\end{equation*}
		as was to be shown.
	\end{proof}
	%
	
	Using Theorems \ref{thm:general_ns} and \ref{mainNijprop} we now show that a Nijenhuis operator on an algebra
	in~$\cC$ leads to a $\cC$-NS-algebra.
	
	\begin{proposition}\label{NijNS}
		Let $(A,\mu)\in\cC$ and let $\fN$ be a Nijenhuis operator for $A$. Define new products on $A$ by setting, for all
		$a,a'\in A$,
		\begin{equation}\label{equ:nij_to_ns}
			a\succ a':=\fN(a)a'\;,\quad    a\prec a':= a\fN(a')\;,\quad    a\vee a':=-\fN(aa')\;.
		\end{equation}
		Then $(A,\prec,\succ,\vee)$ is a $\cC$-NS-algebra.
	\end{proposition}
	\begin{proof}
		We let $M:= A$ and $\alpha:=-\fN\circ\mu$ and $\op:=\fN$ in Theorem \ref{thm:general_ns}
		and verify the assumptions of that theorem. First, the $A$-bimodule structure taken on $A$ is the standard one (see
		Section~\ref{sec1}). The condition that $\op$ ($=\fN$) is an algebra homomorphism is precisely the condition
		\eqref{Nij} that $\fN$ is a Nijenhuis operator. Also, the fact that $(M,\star)$ ($=(A,\star)$) belongs to $\cC$ was
		shown in Theorem~\ref{mainNijprop}.
	\end{proof}
	
	The proposition implies that there is a functor $\catnij\to \catns$, defined on objects by $(A,\fN)\rightsquigarrow
	(A,\prec,\succ,\vee)$, where the latter products are defined by \eqref{equ:nij_to_ns}; on morphisms it is the
	identity.

	\section{Relative Rota-Baxter operators}\label{sec:RRB}
	Relative Rota-Baxter operators have been introduced in their basic form in \cite{Kupershmidt} and have since then
	been generalized to arbitrary operads \cite{pei}. It has been shown in the operadic context that they lead to
	tridendriform algebras. Since, as we have shown in Proposition \ref{prp:tridendri_is_NS}, $\cC$-tridendriform
	algebras are $\cC$-NS-algebras, this shows that relative Rota-Baxter operators also lead to $\cC$-NS-algebras.  We
	give in this section a direct proof of this result, which we state as Proposition \ref{prp:RBB_tridendri}, as a
	direct application of our general result, Theorem \ref{thm:general_ns}. We also prove that relative Rota-Baxter
	operators can be lifted to semidirect sums and derive from it on the one hand an alternative proof that relative
	Rota-Baxter operators lead to dendriform algebras, and on the other hand a construction of bimodules using relative
	Rota-Baxter operators on arbitrary binary algebras; this generalizes a result know for associative, Lie and Leibniz
	algebras, see respectively \cite[Lemma 2.11]{uchino}, \cite[Lemma 3.1]{controlling} and the earlier references
	\cite{bordemann,yvette} cited there, and \cite[Theorem 2.7]{tangshengzhou}.  As before, we prove everything for the
	category $\cC$ of all binary algebras $(A,\mu)$ satisfying a given collection of multilinear relations
	$\cR_1=0,\dots,\cR_k=0$.
	
	\begin{definition}
		Let $(M,\pr)$ be an $A$-bimodule algebra, where $(A,\mu)\in\cC$. Let $\fR:M\rightarrow A$ be a linear map and let
		$\l\in R$.  We say that $\fR$ is a \emph{relative Rota-Baxter operator of weight $\l$ (on $M$)} if
		\begin{equation}
			\fR (m)\fR(m')=\fR(\fR(m)\cdot m'+m\cdot \fR (m')+\l m\pr m')\;, \label{axiomOoperlambda}
		\end{equation}
		for all $m,m'\in M.$ When $(M,\pr)=(A,\mu)$ and the bimodule algebra structure on $M=A$ is the standard one (see
		Example \ref{exa:A_bimodule_algebra}), one says that $\fR$ is a \emph{Rota-Baxter operator of weight $\l$ (on $A$)}.
	\end{definition}
	
	We denote by $\catrrbl$ the category of relative Rota-Baxter operators of weight $\l$. Its objects are
	triplets $(A,M,\fR)$, where $A=(A,\mu)\in\cC$, $M$ is an $A$-bimodule algebra and $\fR$ is a relative Rota-Baxter
	operator of weight $\l$ on $M$. A \emph{morphism} between two relative Rota-Baxter operators $(A,M,\fR)$ and
	$(A',M',\fR')$ is a pair $(\phi,\psi)$, where $\phi:(A,\mu)\rightarrow (A',\mu')$ and $\psi:(M,\op)\to (M',\op')$
	are algebra homomorphisms, satisfying for all $a\in A$ and $m\in M$,
	\begin{equation}
		\phi\circ\fR=\fR'\circ \psi\;, \quad \psi(a\cdot m)=\phi(a)\cdot\psi(m)
		\quad\hbox{and}\quad\psi(m\cdot a)=\psi(m)\cdot\phi(a)\;.
	\end{equation}
	
	%
	
	
	We now prove that every relative Rota-Baxter operator of weight~$\l$ on an $A$-bimodule algebra, with $A$ in $\cC$,
	leads to a $\cC$-tridendriform algebra. This result is very classical for Rota-Baxter operators of weight~$\l$ on
	associative and Lie algebras (see \cite{BaiGuoNew}), but is also known in general in the operadic context, see for
	example \cite{pei}.  For completeness, we give a proof using the notions and notations of the present article. To
	do this, we first show that a relative Rota-Baxter operator on an $A$-bimodule algebra $M$ can be lifted to a
	Rota-Baxter operator on the semidirect sum $A\bowtie M$.
	\begin{proposition}\label{lifttri}
		Let $(A,\mu)\in \cC$, let $(M,\pr)$ be an $A$-bimodule algebra, let $\fR:M\rightarrow A$ be a linear map and let
		$\l\in R$. Define the \emph{lift} $\hat\fR$ of $\fR$:
		\begin{equation}\label{equ:Rhat}
			\hat{\fR}:A\bowtie M\rightarrow A\bowtie M\;, \quad \hat{\fR}(a,m):=(-\l a+\fR (m),0)\;,
		\end{equation}
		for all $a\in A$ and $m\in M$. Then $\fR$ is a relative Rota-Baxter operator of weight $\l$ on $M$ if and
		only if $\hat{\fR }$ is a Rota-Baxter operator of weight $\l $ on $A\bowtie M$.
	\end{proposition}
	\begin{proof}
		For $a,a'\in A$ and $m,m'\in M$, straightforward computations show that
		\begin{equation*}
			\hat{\fR}(a,m)*_{\bullet }\hat{\fR}(a', m')=(\l ^2aa'-\l a\fR (m')-\l \fR (m)a'+\fR (m)\fR (m'), 0)\;,
		\end{equation*}
		and
		\begin{eqnarray*}
			\lefteqn{\hat{\fR}(\hat{\fR}(a, m)*_{\bullet}(a',m')+(a,m)*_\bullet\hat{\fR }(a',m')+\l(a, m)*_\bullet(a',m'))}\\
			&=&(-\l (\fR (m)a'-\l aa'+a\fR (m'))+\fR (\fR (m)\cdot m'+m\cdot \fR (m')+\l m\pr m'), 0)\;.
		\end{eqnarray*}
		Hence, $\hat{\fR }$ is a Rota-Baxter operator of weight $\l$ if and only if
		\begin{equation*}
			\fR (m)\fR (m')=\fR (\fR(m)\cdot m'+m\cdot \fR (m')+\l m\pr m')\;,
		\end{equation*}
		i.e., if and only if $\fR $ is a relative Rota-Baxter operator of weight $\l $.
	\end{proof}
	\begin{proposition}\label{prp:RBB_tridendri}
		Let $(A,\mu)\in\cC$, let $(M,\pr)$ be an $A$-bimodule algebra, let $\l \in R$ and let $\fR :M\rightarrow A$
		be a relative Rota-Baxter operator of weight $\l $. Define products on $M$ by
		\begin{eqnarray*}
			m\prec m':=m\cdot\fR(m')\;,\quad m\succ m':=\fR(m)\cdot m'\;,\quad m\wee m':=\l m\pr m'\;,
		\end{eqnarray*}
		for all $m,m'\in M$. Then $(M,\prec,\succ,\wee)$ is a $\cC$-tridendriform algebra.  
	\end{proposition}
	\begin{proof}
		We prove that $(M\oplus M,\boxtimes)\in\cC$, where $\boxtimes$ is defined for $(m,x),(m',x')\in M\oplus M$ by
		\begin{align}\label{equ:RRB_prod}
			(m,x)\boxtimes (m', x')&:=(m\prec m'+m\succ m'+m.m', m\succ x'+x\prec m'+x\wee x')\\
			&\;=(m\cdot \fR (m')+\fR (m)\cdot m'+\l m\pr m', \fR (m)\cdot x'+x\cdot \fR (m')+\l x\pr x')\nonumber\;.
		\end{align}
		By Proposition \ref{lifttri}, the map $\hat{\fR }$ defined by \eqref{equ:Rhat} is a Rota-Baxter operator of weight
		$\l$ on $A\bowtie M$.  Since $A\bowtie M\in \mathcal{C}$, we can apply Remark 3.5 in \cite{opv}, which implies
		that, if we define, for $(a,m),(a',m')\in A\bowtie M$,
		\begin{align*}
			(a,m)\succ(a',m')&:=\hat{\fR}(a,m)*_{\bullet }(a',m')=(-\l a+\fR(m),0)*_{\bullet }(a',m')\\
			&=(-\l aa'+\fR(m)a',-\l a\cdot m'+\fR(m)\cdot m')\;, \\
			(a,m)\prec(a',m')&=(a,m)*_{\bullet }\hat{\fR}(a',m')=(a,m)*_{\bullet }(-\l a'+\fR(m'),0)\\
			&=(-\l aa'+a\fR(m'),-\l m\cdot a'+m\cdot\fR(m'))\;,\\
			(a,m)\wee(a',m')&=\l (a,m)*_{\bullet }(a',m')=(\l aa',\l a\cdot m'+\l m\cdot a'+\l m\pr m')\;,
		\end{align*}
		then $(A\bowtie M,\prec,\succ,\wee)$ is a $\cC$-tridendriform algebra.  This means that, if we define a
		product~$\hat\boxtimes$ on $(A\bowtie M)\oplus (A\bowtie M)$ by setting, for $a_1,a_1',a_2,a_2\in A$ and
		$m_1,m_1,m_2,m_2'\in M$:
		\begin{eqnarray*}
			\lefteqn{((a_1, m_1), (a_1',m_1'))\hat\boxtimes ((a_2, m_2), (a_2', m_2'))}\\
			&:=&((a_1, m_1)\prec (a_2, m_2)+(a_1, m_1)\succ (a_2, m_2)+(a_1, m_1).(a_2, m_2), \\
			&&\;\ (a_1, m_1)\succ (a_2', m_2')+(a_1', m_1')\prec (a_2, m_2)+(a_1', m_1').(a_2', m_2')) \\
			&=&((a_1\fR (m_2)-\l a_1a_2+\fR (m_1)a_2, m_1\cdot \fR (m_2)+\fR (m_1)\cdot m_2+\l m_1\pr m_2), \\
			&&\;\ (-\l a_1a_2'+\fR (m_1)a_2'-\l a_1'a_2+a_1'\fR (m_2)+\l a_1'a_2', \\
			&&\quad-\l a_1\cdot m_2'+\fR (m_1)\cdot m_2'-\l m_1'\cdot a_2+m_1'\cdot \fR (m_2)+\l a_1'\cdot
			m_2'+\l m_1'\cdot a_2'+\l m_1'\pr m_2'))\;,
		\end{eqnarray*}
		then $((A\bowtie M)\oplus (A\bowtie M),\hat\boxtimes )\in\cC$.  Clearly $(\set0\bowtie M)\oplus (\set0\bowtie M)$ is a
		subalgebra of $((A\bowtie M)\oplus (A\bowtie M), \hat\boxtimes)$, with product
		\begin{eqnarray*}
			\lefteqn{((0,m_1),(0,m_1'))\hat\boxtimes((0,m_2),(0,m_2'))}\\
			&=&((0,m_1\cdot\fR(m_2)+\fR(m_1)\cdot m_2+\l m_1\pr m_2),(0,\fR (m_1)\cdot m_2'+m_1'\cdot\fR (m_2)+\l m_1'\pr m_2'))  \;.
		\end{eqnarray*}
		Comparing this formula with \eqref{equ:RRB_prod}, it is clear that the algebras $((\set0\bowtie
		M)\oplus(\set0\bowtie M),\hat\boxtimes)$ and $(M\oplus M,\boxtimes)$ are isomorphic. Since the first one belongs to
		$\cC$, as a subalgebra of an algebra in~$\mathcal{C}$, this shows that $(M\oplus M,\boxtimes )$ belongs to $\cC$ as
		well.
	\end{proof}
	
	Propositions \ref{prp:tridendri_is_NS} and \ref{prp:RBB_tridendri} imply at once the following result:
	\begin{proposition}\label{RRBNS}
		Let $(A,\mu)\in\cC$, let $(M,\pr)$ be an $A$-bimodule algebra and let $\fR:M\to A$ be a relative Rota-Baxter
		operator of weight $\l$.  On $M$, define the following products, where $m,m'\in M$:
		\begin{equation}\label{equ:rrb_to_ns}
			m\prec m':= m\cdot \fR(m'),\quad m\succ m':=\fR(m)\cdot m',\quad m\vee m':=\l m\pr m'\;.
		\end{equation}
		Then $(M,\prec,\succ,\vee)$ is a $\cC$-NS-algebra.
		\qed
	\end{proposition}

	We give now another proof of this result, using Theorem~\ref{thm:general_ns} rather than Proposition
	\ref{prp:RBB_tridendri}. To do this, let us note first that, if $(A,\mu)\in\cC$ and $(M,\bullet)$ is an
	$A$-bimodule algebra and $\l\in R$, then one can easily see, by using the methods in Section \ref{sec1}, that
	$(M,\l\bullet)$ is also an $A$-bimodule algebra, where by $\l\bullet$ we denote the product on $M$ defined by
	$(\l\bullet)(m\otimes m')=\l m\bullet m'$.
	\begin{lemma}\label{forMstar}
		Let $(A,\mu)\in\cC$, let $(M,\bullet)$ be an $A$-bimodule algebra, let $\l\in R$ and let $\fR:M\rightarrow A$ be a
		linear map. Define $\Gr(\fR):=\{(\fR(m), m)\mid m\in M\}$, the \emph{graph of~$\fR$}, which is an $R$-submodule of
		$A\oplus M$. Then $\fR$ is a relative Rota-Baxter operator of weight~$\l$ if and only if $\Gr(\fR)$ is a subalgebra
		of $(A\oplus M,*_{\l\bullet})$. If this is the case, and we define a product on $M$ by $m\star
		m':=m\cdot\fR(m')+\fR(m)\cdot m'+\l m\pr m'$, for all $m,m'\in M$, then $(M,\star)\in\cC$.
	\end{lemma}
	\begin{proof}
		The first statement follows by a direct computation. For the second, notice that via the inclusion $M\rightarrow
		A\oplus M$, $m\mapsto (\fR(m), m)$, we obtain an algebra isomorphism $(M, \star )\simeq (\Gr(\fR), *_{\lambda
			\bullet })$, and since the latter is in $\mathcal{C}$, as a subalgebra of an algebra in $\mathcal{C}$, it follows
		that $(M, \star )\in \mathcal{C}$ as well.
	\end{proof}
	
	\begin{proof} (alternative proof of Proposition \ref{RRBNS})
		We use Theorem~\ref{thm:general_ns}: it suffices to take $\op:=\fR$ and $\alpha(m\otimes m'):=\l m\pr m'$, for all
		$m,m'\in M$, in the theorem. Then $\op$ is an algebra homomorphism because $\fR$ is a relative Rota-Baxter
		operator, and $(M,\star)\in\cC$ by Lemma \ref{forMstar}.
	\end{proof}
	
	Proposition \ref{RRBNS} implies that there is a functor $\catrrbl\to \catns$, defined on objects by
	$(A,M,\fR)\rightsquigarrow (M,\prec,\succ,\vee)$, where the latter products are defined by \eqref{equ:rrb_to_ns};
	for a morphism $(\phi,\psi):(A,M,\fR)\to (A',M',\fR')$ it is defined by $(\phi,\psi)\rightsquigarrow\psi$, where we
	recall that $\phi:A\to A'$ and $\psi:M\to M'$.
	
	\begin{remark}
		Besides relative Rota-Baxter operators, some other operators are known to lead in the associative case to
		(classical) tridendriform algebras. One may cite for instance the so-called TD-operators introduced in
		\cite{leroux}, and the more general Rota-Baxter operators of weight~$\theta$ (on~$A$, with $A$ associative),
		where $\theta$ is an element of $A$, commuting with all elements in the image of the operator, see
		\cite{kefl}. In view of Proposition \ref{prp:tridendri_is_NS} (or its associative algebra version proved by
		Uchino), these operators lead to (classical) NS-algebras.
	\end{remark}
	
	To finish this section, we present an application of Theorem \ref{mainNijprop} and Proposition \ref{lifttri}.
	
	\begin{proposition}\label{NijRB}
		Let $(A, \mu )\in \mathcal{C}$, let $M$ be an $A$-bimodule and let $\fR:M\rightarrow A$ be a relative Rota-Baxter
		operator of weight zero.  Define a left and a right action of $M$ on $A$ by setting
		\begin{eqnarray*}
			m\triangleright a:=\fR(m)a-\fR(m\cdot a)\;, \quad a\triangleleft m:=a\fR(m)-\fR(a\cdot m)\;,
		\end{eqnarray*}
		for all $a\in A$, $m\in M$. Then $(A, \triangleright , \triangleleft )$ is an $(M, \star )$-bimodule, where
		$m\star m':=\fR(m)\cdot m'+ m\cdot \fR(m')$, for all $m,m'\in M$.
	\end{proposition}
	\begin{proof}
		The proof which we give is inspired by \cite[Remark 3.2]{controlling} and \cite[Remark 3.2]{das}.  We need to prove
		that $(M, \star )\oplus _0(A, \triangleright , \triangleleft )\in \mathcal{C}$. The product on $(M, \star )\oplus
		_0(A, \triangleright , \triangleleft )$ is given by
		\begin{align*}
			(m, a)\diamond (m', a')&=(m\star m', m\triangleright a'+a\triangleleft m')\\
			&=(\fR (m)\cdot m'+m\cdot \fR (m'), \fR (m)a'-\fR (m\cdot a')+a\fR (m')-\fR (a\cdot m')).
		\end{align*}
		By permuting the summands, we need to prove that $(A\oplus M, \boxtimes)\in \cC$, where $\boxtimes$ is the product
		on $A\oplus M$ defined for $(a,m),(a',m')\in A\oplus M$ by
		\begin{equation*}
			(a,m)\boxtimes (a',m'):=(\fR(m)a'-\fR(m\cdot a')+a\fR(m')-\fR(a\cdot m'),\fR(m)\cdot m'+m\cdot\fR(m'))\;.
		\end{equation*}
		To do this, we consider the linear map $\fN:A\oplus_0M\rightarrow A\oplus_0M$ defined by
		$\fN(a,m):=(\fR(m),0)$, which is, according to Proposition \ref{lifttri} with $\l=0$, a Rota-Baxter
		operator of weight zero for $A\oplus_0M$; since $\fN^2=0$, $\fN$ is a Nijenhuis operator for
		$(A\oplus_0M, *)\in\cC$. By Theorem
		\ref{mainNijprop}, $(A\oplus M, \widetilde\boxtimes)\in \cC$, where $\widetilde\boxtimes$ is the product on
		$A\oplus M$ defined by
		\begin{equation*}
			(a, m)\widetilde\boxtimes (a',m'):=\fN(a,m)*(a',m')+(a, m)*\fN(a',m')-\fN((a,m)*(a',m'))\;.
		\end{equation*}
		In view of the above definition of $\fN$ and the definition of $*$, we can compute:
		\begin{align*}
			(a, m)\widetilde\boxtimes (a',m')
			&=(\fR(m),0)*(a', m')+(a, m)*(\fR (m'), 0)-\fN(aa', a\cdot m'+m\cdot a')\\
			&=(\fR (m)a', \fR (m)\cdot m')+(a\fR (m'), m\cdot \fR (m'))-(\fR (a\cdot m')+\fR (m\cdot a'), 0)\\
			&=(\fR (m)a'-\fR (m\cdot a')+a\fR (m')-\fR (a\cdot m'), \fR (m)\cdot m'+m\cdot \fR (m'))\\
			&=(a, m)\boxtimes (a',m')\;.
		\end{align*}
		It follows that $\widetilde\boxtimes =\boxtimes $ and so
		$(A\oplus M, \boxtimes )\in \cC$, finishing the proof.
	\end{proof}
	\begin{remark}
		Proposition \ref{NijRB} admits a generalization to twisted Rota-Baxter operators (to be defined in the next
		section) instead of relative Rota-Baxter operators of weight zero (generalizing in turn
		\cite[Proposition~3.5]{dastwistedLie}, \cite[Proposition~3.3]{dastwistedassoc} and
		\cite[Proposition~3.1]{dasleibniz}), but the above proof does not seem to work anymore and a more complicated proof
		is needed (this will not be included here because it is beyond the topics of the present paper).
	\end{remark}
	
	\section{Twisted Rota-Baxter operators}\label{sec2}
	In this section we introduce the definition of a twisted Rota-Baxter operator on a general binary algebra
	$(A,\mu)\in\cC$ and show in Proposition \ref{TRBNS} by using Theorem \ref{thm:general_ns} that such an operator
	leads also to a $\cC$-NS-algebra. We thereby generalize a result which was already proven in the cases of
	associative, Lie and Leibniz algebras, respectively in \cite{uchino}, \cite{dastwistedLie} and
	\cite{dasleibniz}. We show in Proposition \ref{prp:NS_to_TRB} that conversely every $\cC$-NS-algebra is obtained
	this way, thereby again generalizing a result already proven for associative, Lie and Leibniz algebras, in the
	cited papers. We deduce from it in Corollary \ref{cor:char} a characterization of $\cC$-NS-algebras, generalizing a
	result that was previously only known in the Leibniz algebra case \cite[Proposition 5.7]{dasleibniz}.  As in the
	previous sections, $\cC$ denotes the category of all binary algebras $(A,\mu)$ satisfying a given collection of
	multilinear relations $\cR_1=0$, \dots, $\cR_k=0$.
	
	\begin{definition}\label{deftRBO}
		Let $(A,\mu)\in\cC$, let $M$ be an $A$-bimodule, and suppose that $H$ is a 2-cocycle on~$A$ with values in $M$.  A
		linear map $\fT :M\to A$ is called an \emph{$H$-twisted Rota-Baxter operator (on~$M$)} if, for all $m,m'\in M$,
		\begin{equation}\label{twistedRBO}
			\fT (m)\fT (m')=\fT (\fT (m)\cdot m'+m\cdot \fT (m')+H(\fT (m)\otimes \fT (m')))\;.
		\end{equation}
	\end{definition}
	
	Clearly, an $H$-twisted Rota-Baxter operator on $M$ for which $H$ is trivial is just a relative Rota-Baxter
	operator of weight zero on $M$. Therefore, twisted Rota-Baxter operators generalize relative Rota-Baxter operators
	of weight zero.
	
	\begin{example}\label{exa:reynolds}
		Proposition \ref{NijNS}, applied to $(A,\mu)\in\cC$ and $\fN:=\Id_A$ shows that $(A,\mu,\mu,-\mu)$ is a
		$\cC$-NS-algebra. As we will see in Proposition \ref{prp:NS_to_TRB} below, this implies that $-\mu$ is a 2-cocycle
		on $A$. A $(-\mu)$-twisted Rota-Baxter operator on $(A,\mu)$ is called a Reynolds operator (see \cite{rota} for the
		associative algebra case).  So a \emph{Reynolds operator (on $A$)} is a linear map $\op:A\rightarrow A$ satisfying
		for all $a,a'\in A$,
		\begin{equation*}
			\op(a)\op(a')=\op(\op(a)a'+a\op(a')-\op(a)\op(a'))\;.
		\end{equation*}
	\end{example}
	
	We denote by $\cattrb$ the category of twisted Rota-Baxter operators. The objects of $\cattrb$ are quadruplets
	$(A,M,H,\fT)$, where $A=(A,\mu)\in\cC$, where $M$ is an $A$-bimodule, $H$ is a 2-cocycle on~$A$ with values in $M$,
	and $\fT$ is an $H$-twisted Rota-Baxter operator on $M$. A \emph{morphism} between two twisted Rota-Baxter
	operators $(A,M,H,\fT)$ and $(A',M',H',\fT')$ consists of a pair $(\phi,\psi)$, where $\phi:(A,\mu)\rightarrow
	(A',\mu')$ is an algebra homomorphism and $\psi:M\rightarrow M'$ is a linear map satisfying, for all $a\in A$ and
	$m\in M$,
	\begin{eqnarray*}
		&&\phi \circ \fT=\fT'\circ \psi\;, \\
		&&\psi (a\cdot m)=\phi (a)\cdot \psi (m) \quad\hbox{and}\quad \psi (m\cdot a)=\psi (m)\cdot\phi(a)\;, \\
		&&\psi \circ H=H'\circ (\phi\otimes \phi )\;.
	\end{eqnarray*}
	
	We first prove a result which we will use to show that twisted Rota-Baxter operators lead to $\cC$-NS-algebras.
	\begin{proposition}\label{Mstar}
		Let $(A,\mu)\in\cC$, let $M$ be an $A$-bimodule, and let $H$ be a 2-cocycle on~$A$ with values in $M$.  Let
		$\fT:M\rightarrow A$ be an $H$-twisted Rota-Baxter operator on $M$. Define a product $\star$ on $M$ by
		\begin{equation}
			m\star m':=\fT(m)\cdot m'+m\cdot \fT(m')+H(\fT(m)\otimes \fT(m'))\;,
		\end{equation}
		for all $m, m'\in M$. Then $(M,\star)\in\cC$.
	\end{proposition}
	\begin{proof}
		We consider as in Lemma \ref{forMstar} the graph of $\fT :M\rightarrow A$, which is the $R$-submodule
		$\Gr(\fT):=\{(\fT(m), m)\mid m\in M\}$ of $A\oplus M$. It follows at once from Definition \ref{deftRBO} that
		since $\fT$ is an $H$-twisted Rota-Baxter operator, $\Gr(\fT)$ is a subalgebra of $A\oplus_HM$, so that
		$(\Gr(\fT),*_H)\in\cC$. The linear map $M\to \Gr(\fT)$, defined by $m\mapsto (\fT(m),m)$ is an $R$-module
		isomorphism, which is an algebra isomorphism $(M,\star)\to (\Gr(\fT),*_H)$. Since $(\Gr(\fT),*_H)$ belongs to
		$\cC$, so does $(M,\star)$, as was to be shown.
	\end{proof}
	
	It follows from the proposition that we have a functor $\cattrb\to\cC$, which is defined on objects by
	$(A,M,H,\fT)\rightsquigarrow M=(M,\star)$ and on morphisms by $(\phi,\psi)\rightsquigarrow \psi$. Using the
	proposition we show that twisted Rota-Baxter operators lead to $\cC$-NS-algebras.
	
	\begin{proposition}\label{TRBNS}
		Let $(A,\mu)\in\cC$, let $M$ be an $A$-bimodule and let $H$ be a $2$-cocycle on $A$ with values in $M$. Also, let
		$\fT:M\rightarrow A$ be an $H$-twisted Rota-Baxter operator. Then $(M,\prec,\succ,\vee)$ is a $\cC$-NS-algebra
		with the products given for all $m,m'\in M$ by:
		\begin{equation}\label{equ:trb_to_ns}
			m\prec m':=m\cdot \fT(m')\;,\quad m\succ m':=\fT(m)\cdot m'\;,\quad m\vee m':=H(\fT(m)\otimes\fT(m'))\;.
		\end{equation}
	\end{proposition}
	\begin{proof}
		In Theorem~\ref{thm:general_ns} we now let $\op:=\fT$ and define $\alpha$ by $\alpha(m\otimes m'):=
		H(\fT(m)\otimes\fT(m'))$ for all $m,m'\in M$. Then $\op$ is an algebra homomorphism because $\fT$ is an $H$-twisted
		Rota-Baxter operator (see \eqref{twistedRBO}). According to Proposition \ref{Mstar}, $(M,\star)\in\cC$. Thus, the
		conditions of Theorem~\ref{thm:general_ns} are satisfied and $(M,\prec,\succ,\vee)$ is a $\cC$-NS-algebra.
	\end{proof}

	Once more, the proposition leads to a functor $F:\cattrb\to\catns$, which is defined on objects by
	$(A,M,H,\fT)\rightsquigarrow(M,\prec,\succ,\vee)$ where the latter products are defined by
	\eqref{equ:trb_to_ns}. On morphisms $(\phi,\psi):(A,M,H,\fT)\to(A',M',H',\fT')$ it is given by
	$(\phi,\psi)\rightsquigarrow\psi$, where we recall that $\phi:A\to A'$ and $\psi:M\to M'$.
	
	We now show that every $\cC$-NS-algebra can be obtained by Proposition \ref{TRBNS}.
	
	\begin{proposition}\label{prp:NS_to_TRB}
		Let $(A,\prec,\succ,\vee)$ be a $\cC$-NS-algebra and let, as before, $\star$ denote the sum of the products
		$\prec,\ \succ$ and $\vee$ on $A$. Define $H:A\otimes A\rightarrow A$ by $H(a\otimes a'):=a\vee a'$, for all $a,
		a'\in A$. Then $H$ is a 2-cocycle on $(A,\star)$ with values in the bimodule $(A,\succ,\prec)$ and
		$\Id_A:(A,\succ,\prec)\rightarrow (A,\star)$ is an $H$-twisted Rota-Baxter operator on $A$.
	\end{proposition}
	\begin{proof}
		The second statement follows at once from the definitions of $\star$ and of an $H$-twisted Rota-Baxter operator
		\eqref{twistedRBO}. Therefore, we only have to prove that $H$ is a 2-cocycle; said differently, $(A\oplus A,*_H)\in
		\cC$, where for all $a,a',x,x'\in A$:
		\begin{equation*}
			(a,x)*_H (a',x')=(a\star a',x\prec a'+a\succ x'+a\vee a')\quad\hbox{and}\quad
			a\star a'=a\prec a'+ a\succ a'+a\vee a'\;.
		\end{equation*}
		By assumption, $(A,\star)\in\cC$ and $(A,\succ,\prec)$ is an $(A,\star)$-bimodule, i.e., $(A\oplus A,\bt)\in \cC$,
		where for all $a,a',x,x'\in A$:
		\begin{equation*}
			(a,x)\bt(a',x')=(a\star a',x\prec a'+a\succ x')\;.
		\end{equation*}
		We denote as before $A_0=A\oplus \{0\}$, $A_1=\{0\}\oplus A$, $\underline{a}_0=(a, 0)$, $\underline{a}_1=(0, a)$,
		for $a\in A$. With this notation, the products $\bt$ and $*_H$ on $A\oplus A$ are given, for all $a,a'\in A$, by :
		\begin{align}\label{equ:products}
			\a0\bt\ul{a'}0&=\ul{a\star a'}0\;,&\a0\bt\ul{a'}1&=\ul{a\succ a'}1\;,&\a1\bt\ul{a'}0&=\ul{a\prec
				a'}1\;,&\a1\bt\ul{a'}1&=0\;,\nonumber\\
			\a0*_H\ul{a'}0&=\ul{a\star a'}0+\ul{a\vee a'}1\;,&\a0*_H\ul{a'}1&=\ul{a\succ a'}1\;,&\a1*_H\ul{a'}0&=\ul{a\prec a'}1\;,
			&\a1*_H\ul{a'}1&=0\;.
		\end{align}
		Notice the close similarity of these products. In fact, let $X=a_1a_2\dots a_n$ be a parenthesized monomial of
		length $n$ of $A$ and denote for $1\leqslant i\leqslant n$,
		\begin{equation}\label{equ:notation}
			X_\bt^i= \ul{a_1}0\bt\cdots \bt\ul{a_{i-1}}0\bt\ul{a_{i}}1  \bt\ul{a_{i+1}}0\bt\cdots\bt\ul{a_n}0\;,  \quad\hbox{and}\quad
			X_\bt^0= \ul{a_1}0\bt\ul{a_{2}}0\bt\cdots\bt\ul{a_n}0\;,  
		\end{equation}
		and similarly for $X_{*_H}^i$ and $X_{*_H}^0$. We show by induction on $n$ that
		\begin{equation}\label{equ:proof-NSTRB}
			X_{*_H}^i=X_\boxtimes^i\;,\quad\hbox{ for }\quad i=1,\dots,n\;.
		\end{equation}
		For $n=2$ this follows at once from \eqref{equ:products}, so let us assume that the property is true for all values
		smaller than some $n>2$. We write $X=YZ$ as (uniquely) determined by the parenthesizing and use the notation
		\eqref{equ:notation} also for $Y$ and $Z$. Notice that $Z_{*_H}^0=Z_\bt^0+W$ where $W\in A_1$. When $i$ is at most
		the length of $Y$ we find upon using the induction hypothesis \eqref{equ:proof-NSTRB} and \eqref{equ:products} that
		\begin{equation*}
			X_{*_H}^i=Y_{*_H}^i*_HZ_{*_H}^0\overset{\eqref{equ:proof-NSTRB}}=Y_\bt^i*_H(Z_\bt^0+W)
			\overset{\eqref{equ:products}}=Y_\bt^i\bt Z_\bt^0=X_\bt^i\;.
		\end{equation*}
		The case when $i$ is larger than the length of $Y$ follows by symmetry.
		
		Let $\cR$ be an $n$-linear relation of $\cC$. We have to prove that $\cR_{*_H}=0$. Like in the proof of Theorem
		\ref{thm:general_ns}, it is enough to prove that $\cR_{*_H}(u_1, \dots, u_n)=0$ when the elements $u_i$ are
		in~$A_0$ or $A_1$, with at most one element in $A_1$. If one element is in $A_1$, then
		$\cR_{*_H}(u_1,\dots,u_n)=\cR_{\bt}(u_1,\dots,u_n)$, since according to \eqref{equ:proof-NSTRB}, $X_{*_H}^i=X_\bt^i$
		for any monomial $X$ of $A$. Since $\cR_\bt=0$ this shows that $\cR_{*_H}(u_1,\dots,u_n)=0$ when one element $u_i$ is
		in $A_1$. We therefore only need to analyze the case when all the elements $u_i$ are in $A_0$, say $u_i=\ul{a_i}0$
		with $a_i\in A$ for $i=1,\dots,n$. We will prove the following formula by induction on $n$:
		\begin{equation}\label{equ:*_H_to_star}
			\ul{a_1}0*_H\cdots*_H\ul{a_n}0=\ul{a_1\star\cdots\star a_n}0+\ul{a_1\star\cdots\star a_n}1
			-\sum_{i=1}^n\ul{a_1}0\bt\cdots\bt\ul{a_i}1\bt\cdots\bt\ul{a_n}0\;,
		\end{equation}
		which is written in terms of the above notations as
		\begin{equation}\label{equ:*_H_to_star_bis}
			X_{*_H}^0=\ul{X_\star}0+\ul{X_\star}1-\sum_{i=1}^nX_\bt^i\;.
		\end{equation}

		Since $(A,\star)\in \cC$ and $(A\oplus A,\bt)\in \cC$, we obtain from it that
		$\displaystyle\cR_{*_H}=\ul{\cR_\star}0+\ul{\cR_\star}1-\sum_{i=1}^n\cR_\bt^i=0$ for elements of $A_0$, where
		$\cR_\bt^i$ is defined similarly to $X_\bt^i$ above, thereby finishing the proof.
		
		We first consider \eqref{equ:*_H_to_star} when $n=2$. Then $X=a_1a_2$ and \eqref{equ:*_H_to_star} reads
		\begin{equation*}
			\ul{a_1}0*_H\ul{a_2}0=\ul{a_1\star a_2}0+\ul{a_1\star a_2}1-\ul{a_1}1\bt\ul{a_2}0-\ul{a_1}0\bt\ul{a_2}1\;,
		\end{equation*}
		and its validity is clear from \eqref{equ:products}.
		
		Suppose now that~\eqref{equ:*_H_to_star} is true for all monomials of length $n-1$. Let $\ell<n$ be such that
		$X=a_1\dots a_{\ell-1}(a_\ell a_{\ell+1})a_{\ell+2}\dots a_n$, with some extra parenthesizing. Notice that such an
		index $\ell$ is not unique, in general. By using \eqref{equ:proof-NSTRB}, the induction hypothesis
		\eqref{equ:*_H_to_star} and \eqref{equ:products} (twice) we find
		\begin{eqnarray*}
			X_{*_H}^0
			&=&\ul{a_1}0*_H\cdots*_H\ul{a_\ell\star a_{\ell+1}}0*_H\cdots*_H\ul{a_{n}}0+
			\ul{a_1}0*_H\cdots*_H\ul{a_\ell\vee a_{\ell+1}}1*_H\cdots*_H\ul{a_{n}}0\\
			&\overset{\eqref{equ:proof-NSTRB}}=&\ul{a_1}0*_H\cdots*_H\ul{a_\ell\star a_{\ell+1}}0*_H\cdots*_H\ul{a_{n}}0+
			\ul{a_1}0\bt\cdots\bt\ul{a_\ell\vee a_{\ell+1}}1\bt\cdots\bt\ul{a_{n}}0\\
			&\overset{\eqref{equ:*_H_to_star}}=&\ul{a_1\star\cdots\star (a_\ell\star a_{\ell+1})\star\cdots\star a_{n}}0+
			\ul{a_1\star\cdots\star (a_\ell\star a_{\ell+1})\star\cdots\star a_{n}}1\\
			&&\qquad-\sum_{{i=1}\atop{i\neq\ell,\ell+1}}^{n}\ul{a_1}0\bt\cdots\bt\ul{a_i}1\bt\dots\bt\ul{a_\ell\star
				a_{\ell+1}}0\bt\cdots\bt\ul{a_{n}}0\\
			&&\qquad -\ul{a_1}0\bt\dots\bt\ul{a_\ell\star a_{\ell+1}}1\bt\dots\bt\ul{a_{n}}0+
			\ul{a_1}0\bt\dots\bt\ul{a_\ell\vee a_{\ell+1}}1\bt\dots\bt\ul{a_{n}}0\\
			&\overset{\eqref{equ:products}}=&\ul{X_\star}0+\ul{X_\star}1
			-\sum_{{i=1}\atop{i\neq\ell,\ell+1}}^{n}\ul{a_1}0\bt\cdots\bt\ul{a_i}1\bt\dots\bt(\ul{a_\ell}0\bt\ul{a_{\ell+1}}0)\bt
			\cdots\bt\ul{a_{n}}0\\
			&&\qquad -\ul{a_1}0\bt\dots\bt\ul{a_\ell\prec a_{\ell+1}}1\bt\dots\bt\ul{a_{n}}0-
			\ul{a_1}0\bt\dots\bt\ul{a_\ell\succ a_{\ell+1}}1\bt\dots\bt\ul{a_{n}}0\\
			&\overset{\eqref{equ:products}}=&\ul{X_\star}0+\ul{X_\star}1-\sum_{i=1}^nX_\bt^i\;,
		\end{eqnarray*}
		which proves \eqref{equ:*_H_to_star_bis}.
	\end{proof}
	\begin{remark}\label{forexamples}
		As a byproduct of the proof, we find that
		\begin{equation}\label{equ:alt_formula}
			\ul{\cR_\star}1-\sum_{i=1}^n\cR_\bt^i=\cR_{*_H}-\ul{\cR_\star}0\;.
		\end{equation}
		Both sides belong to $A_1$ and the left hand side is the difference between the relation $\cR_\star=0$ and the
		relations obtained by substituting in $\cR_\boxtimes$ one element from $A_1$ and all other elements from $A_0$, so
		it is the extra relation that we obtained in the Examples \ref{exa:ns_assoc} -- \ref{exa:NAP}. It then follows
		from~\eqref{equ:alt_formula} that this relation can also be obtained as the difference between evaluating~$\cR$
		with the product $*_\vee=*_H$ and with the product $\star$, on elements of $A_0\simeq A$.
	\end{remark}
	It follows from the proposition that we have a functor $G:\catns\rightarrow \cattrb$ which is defined on objects by
	$(A,\prec,\succ,\vee)\rightsquigarrow(A, (A, \succ , \prec ), \vee,\Id_A)$. On morphisms $f:A\rightarrow A'$ it is given by
	$f\rightsquigarrow(f,f)$.
	
	We defined earlier using Proposition \ref{TRBNS} the functor $F:\cattrb\to\catns$.  By construction, $F\circ G$ is
	the identity functor on $\catns$. Moreover, one can easily verify that the pair $(G,F)$ is an adjunction: we have
	an isomorphism of bifunctors
	$$
	\Hom_{\cattrb}(G(\bullet),\bullet)\simeq \Hom_{\catns}(\bullet,F(\bullet))\;.
	$$
	
	Note that, by combining Propositions \ref{NijNS} and \ref{prp:NS_to_TRB}, we obtain immediately:
	\begin{corollary}
		Let $(A, \mu )\in \mathcal{C}$ and let $\fN:A\rightarrow A$ be a Nijenhuis operator. For $a,a'\in A$ define $a\star
		a':=\fN(a)a'+a\fN(a')-\fN(aa')$, $a\prec a':=a\fN(a')$, $a\succ a':=\fN(a)a'$. Define $H:A\otimes A\rightarrow A$,
		$H(a\otimes a'):=-\fN(aa')$. Then $(A,\star)\in\cC$, $(A,\succ,\prec)$ is an $(A,\star)$-bimodule and $H$ is a
		2-cocycle on $(A,\star )$ with values in $(A,\succ ,\prec)$.\qed
	\end{corollary}
	
	We end with the following result, which follows easily from Propositions \ref{TRBNS} and \ref{prp:NS_to_TRB}.
	\begin{corollary}\label{cor:char}
		Let $(A,\star)\in \cC$. Then there exists a $\cC$-NS-algebra $(A,\prec,\succ,\vee)$ such that $\star =\prec +\succ
		+\vee $ if and only if there exists an $A$-bimodule $M$, a 2-cocycle $H:A\otimes A\rightarrow M$ and a bijective
		$H$-twisted Rota-Baxter operator $\fT:M\rightarrow A$.\qed
	\end{corollary}
	
	\bigskip
	
	\noindent \textbf{Acknowledgments:} During the final stage of the writing of this paper, the second author was
	partially supported by a grant from UEFISCDI, project number PN-III-P4-PCE-2021-0282.

	\bibliographystyle{amsplain}
	\providecommand{\bysame}{\leavevmode\hbox to3em{\hrulefill}\thinspace}
	\providecommand{\MR}{\relax\ifhmode\unskip\space\fi MR }
	\providecommand{\MRhref}[2]{%
		\href{http://www.ams.org/mathscinet-getitem?mr=#1}{#2}
	}
	\providecommand{\href}[2]{#2}

\end{document}